\documentclass[letterpaper,11pt]{article}

\usepackage[utf8]{inputenc}
\usepackage{amssymb,amsmath,amsthm,mathtools,amsbsy,amsfonts}
\usepackage{subcaption}
\usepackage{graphicx,wasysym}

\usepackage[margin=1.15in]{geometry}
\usepackage{stmaryrd}
\usepackage{cases}
\usepackage{empheq}
\usepackage{paralist}
\usepackage{multimedia}
\usepackage{tikz}
\usepackage{enumitem}
\usepackage{cite}
\usepackage{multirow}
\usepackage{accents}
\usepackage{url}

\usepackage{fullpage}

\usepackage[explicit]{titlesec}
\titleformat{\section}{\large\bfseries\center\raggedright}{\thesection}{0.5em}{{#1}}[]
\titleformat{\subsection}{\bfseries\center\raggedright}{\thesubsection}{0.5em}{{#1}}[]
\titleformat{\subsubsection}[runin]{\bfseries}{\thesubsubsection}{0.5em}{{#1}}[.]
\titlespacing*{\section}{0pt}{0.8\baselineskip}{0.6\baselineskip}
\titlespacing*{\subsection}{0pt}{0.6\baselineskip}{0.4\baselineskip}
\titlespacing*{\subsubsection}{0pt}{0.4\baselineskip}{0.4\baselineskip}

\newcommand\namefont{\normalfont\bfseries}
\newcommand\numberfont{\normalfont\bfseries}
\newcommand\notefont{\normalfont\bfseries}
\newtheoremstyle{mystyle} 
	{0.3em} 
	{0.3em} 
	{\itshape} 
	{} 
	{\normalfont} 
	{.} 
	{.5em} 
	{{\namefont\thmname{#1}}~{\numberfont\thmnumber{#2}}{\notefont\thmnote{ (#3)}}} 
\theoremstyle{plain}
\newtheorem{thm}{Theorem}[section]

\newtheorem{rem}[thm]{Remark}
\newtheorem{lem}[thm]{Lemma}

\newtheorem{defn}[thm]{Definition}
\newtheorem{notn}[thm]{Notation}

\newtheoremstyle{namedthmstyle} 
        {} 
	{} 
	{\itshape} 
	{} 
	{\bfseries} 
	{} 
	{ } 
        {}
\theoremstyle{namedthmstyle}
\newcommand{\thistheoremname}{}
\newtheorem*{genericthm}{\thistheoremname}
\newenvironment{namedthm}[1]
{\renewcommand{\thistheoremname}{#1}%
  \begin{genericthm}}
  {\end{genericthm}}
\makeatletter
\def\namedlabel#1#2{\begingroup
   \def\@currentlabel{#2}%
   \label{#1}\endgroup
}
\makeatother

\allowdisplaybreaks
\mathtoolsset{showonlyrefs}
\renewcommand{\leq}{\leqslant}
\renewcommand{\geq}{\geqslant}

\newcommand{\ul}[1]{\underline{\smash{#1}}}

\definecolor{darkred}{rgb}{0.6,0.1,0.1}
\definecolor{darkgreen}{rgb}{0.1,0.6,0.1}
\definecolor{darkblue}{rgb}{0.1,0.1,0.6}

\newcommand{\be}{\begin{equation}}
\newcommand{\ee}{\end{equation}}
\newcommand{\bes}{\begin{equation*}}
\newcommand{\ees}{\end{equation*}}
\newcommand{\bfig}{\begin{figure}}
\newcommand{\efig}{\end{figure}}
\newcommand{\bt}{\begin{table}}
\newcommand{\et}{\end{table}}
\newcommand{\bc}{\begin{center}}
\newcommand{\ec}{\end{center}}

\newcommand{\mt}[1]{\mathrm{#1}}
\def\st{\, \left|\right. \,}
\def\:{\colon}

\newcommand{\overbar}[1]{\mkern 1.7mu\overline{\mkern-3.7mu#1\mkern-1.7mu}\mkern 1.7mu} 

\newcommand{\abs}[1]{\left\vert#1\right\vert}
\newcommand{\ap}[1]{\left\langle#1\right\rangle}
\newcommand{\norm}[1]{\left\Vert#1\right\Vert}

\def\grad{\nabla}

\def\conn{\nabla}

\DeclareMathOperator{\dist}{dist}
\DeclareMathOperator{\supp}{supp}

\DeclareMathOperator{\Diam}{diam}

\DeclareMathOperator{\Hess}{Hess}

\def\R{\mathbb{R}} 
\def\P{\mathcal{P}} 
\def\M{M} 
\def\Cont{\mt{C}} 

\def\Lip{\mt{Lip}} 

\def\K{\mathcal{K}}

\def\e{\varepsilon}

\def\d{\,\mathrm{d}}

\def\XXint#1#2#3{{\setbox0=\hbox{$#1{#2#3}{\int}$ }
	\vcenter{\hbox{$#2#3$ }}\kern-.6\wd0}}

\def\der{\mathrm{d}}

\def\p{\partial}

\def\bd{\boldsymbol{W}}

\def\V{\mathcal{V}} 

\def\dim{n}
\def\V{v}
\def\PsiX{\Psi^t_X}
\def\PsiY{\Psi^t_Y}

\def\Lip{\Lambda} 

\def\diam{\Delta}

\def\iM{\mathrm{inj}(\M)} 
\def\cM{\mathrm{conv}(\M)}

\begin{document}

\title{Well-posedness of an interaction model on Riemannian manifolds}

\author{Razvan C. Fetecau\footnote{Department of Mathematics, Simon Fraser University, Burnaby, BC V5A 1S6, Canada.\newline \indent\indent {\protect\url{razvan_fetecau@sfu.ca}}} \and Francesco S. Patacchini\footnote{IFP Energies nouvelles, 1-4 avenue de Bois-Préau, 92852 Rueil-Malmaison, France.\newline \indent\indent {\protect\url{francesco.patacchini@ifpen.fr}}}}

\date{\today}

\maketitle


\begin{abstract}
  We investigate a model for collective behaviour with intrinsic interactions on smooth Riemannian manifolds. For regular interaction potentials, we establish the local well-posedness of measure-valued solutions defined via optimal mass transport. We also extend our result to the global well-posedness of solutions for manifolds with nonpositive bounded sectional curvature. The core concept underlying the proofs is that of Lipschitz continuous vector fields in the sense of parallel transport.
\end{abstract}

{\small

}


\section{Introduction}
\label{sect:intro}
This paper is concerned with the following integro-differential equation for the evolution of a population density $\rho$ on a Riemannian manifold $\M$: 
\begin{equation}
 \label{eqn:model}
\partial_t \rho-\nabla_M \cdot(\rho \nabla_\M K\ast\rho)=0,
\end{equation}
where $K\: \M\times \M \to \R$ is an interaction (also known as aggregation) potential which models social interactions such as attraction and repulsion, and $\nabla_\M \cdot$ and $\nabla_\M $ represent the manifold divergence and gradient, respectively. In \eqref{eqn:model} the symbol $\ast$ denotes a measure convolution: for a time-dependent measure $\rho_t$ on $\M$ and $x\in\M$ we set
\begin{equation} \label{eqn:conv}
	K * \rho_t(x) = \int_M K(x,y) \d \rho_t(y).
\end{equation}
We restrict our solutions $\rho$ to be probability measures on $\M$ at all times: $\int_M  \d \rho_t=1$ for all $t$.

Model \eqref{eqn:model} has numerous applications in swarming and self-organized behaviour in biology \cite{M&K}, material science \cite{CaMcVi2006}, robotics \cite{Gazi:Passino,JiEgerstedt2007}, and social sciences \cite{MotschTadmor2014}. In such applications, equation \eqref{eqn:model} can model interactions between biological organisms such as insects, birds or cells, as well as interactions between robots or even opinions. Concerning the theory, the mathematical analysis of solutions to model \eqref{eqn:model} has focused almost exclusively on the model set up on Euclidean space $\R^\dim$; we refer in this case to \cite{BodnarVelasquez2, BertozziLaurent, Figalli_etal2011, BeLaRo2011} for the well-posedness of the initial-value problem and to \cite{LeToBe2009, FeRa10, BertozziCarilloLaurent, FeHuKo11,FeHu13} for the long-time behaviour of solutions.

The goal of the paper is to establish the well-posedness of measure-valued solutions to model \eqref{eqn:model} set up on general Riemannian manifolds. There are only very few works on this subject, and all require that the manifold be embedded in a larger Euclidean space. For instance, in \cite{WuSlepcev2015,CarrilloSlepcevWu2016,PaSl2021}, the authors investigate the well-posedness of the aggregation model \eqref{eqn:model} when the interactions are {\em extrinsic}, in the sense that the interaction potential depends on the embedding Euclidean distance between points. Another class of models considers {\em intrinsic} interactions, where the interaction potential depends on the geodesic distance on $\M$ between points. Recent works in this direction investigate the well-posedness of model \eqref{eqn:model} with intrinsic interactions on the sphere \cite{FePaPa2020} and on the special orthogonal group $SO(3)$ \cite{FeHaPa2021}. In both \cite{FePaPa2020} and \cite{FeHaPa2021}, however, certain calculations make use of the ambient vector spaces of the manifolds, i.e., in $\R^\dim$ and $\R^{3 \times 3}$, respectively. 

This paper presents a fully intrinsic approach to the well-posedness of solutions to \eqref{eqn:model} on Riemannian manifolds, which does not require any extrinsic calculations in an ambient vector space. To this aim we use the concept of Lipschitz continuity of vector fields via parallel transport. While this is a concept widely used in the literature on optimization on manifolds \cite{OliveiraFerreira2020, FeLoPr2019}, it is much less common in the analysis of differential equations on manifolds. Lipschitz continuity by parallel transport enables us to compare tangent vectors in an intrinsic manner. This approach fundamentally distinguishes itself from that in  \cite{FePaPa2020} and \cite{FeHaPa2021}, where tangent vectors at different points on the manifold are compared in the norm of an ambient vector space. 

As a further motivation for this paper we also mention various studies on the long-time behaviour of solutions to \eqref{eqn:model} on manifolds. An interesting collection of equilibria on the sphere and the hyperbolic plane for model \eqref{eqn:model} with intrinsic interactions can be found in \cite{FeZh2019}. These equilibria show a very rich pattern formation behaviour (e.g., disks, annuli, rings), similar to what has been observed in $\R^\dim$ \cite{KoSuUmBe2011, Brecht_etal2011}. In addition, emergent behaviour has been studied extensively in the related Lohe-type models with extrinsic interactions on the unit sphere, matrix manifolds and tensor spaces in \cite{HaKoRy2017,HaKoRy2018,HaPark2020}. In these works the focus is to investigate the formation of consensus solutions, where the equilibria consist of an aggregation at a single point, rather than the well-posedness of \eqref{eqn:model} in the intrinsic setting.

The paper is structured as follows. We first review in Section \ref{sect:prelim} the concept of parallel transport on a Riemannian manifold and that of Lipschitz continuity by parallel transport, and present useful relations in terms of Hessians and flow maps. In Section \ref{sect:interaction-eq} we give the rigorous notion of solution to equation \eqref{eqn:model} which will be used in the well-posedness theory, as well as introduce the $1$-Wasserstein probability space, which determines the regularity of the solutions we seek. Then, in Section \ref{sect:local-wp} we present our proof of local and global well-posedness for equation \eqref{eqn:model}, starting with the fundamental lemmas underlying the proof. In this section, our results on the global well-posedness are restricted to manifolds with nonpositive bounded sectional curvature. Finally, in Section \ref{sect:add-results} we state two additional results which can be obtained by simply adapting results already proven for the sphere in \cite{FePaPa2020}. Appendices \ref{app:terminology} and \ref{app:prel} give, respectively, basic notions of differential geometry useful to our purposes, and the proofs to preliminary results stated in the main body of the paper.

Everywhere in this paper, $\M$ denotes a manifold satisfying the assumption that follows. We shall only refer to it in some of the main statements; anywhere else, it will be implicitly assumed.  
\begin{namedthm}{(M)}
\namedlabel{hyp:M}{{\bfseries \upshape (M)}}
  $\M$ is a complete, simply connected, smooth Riemannian manifold of finite dimension $n$, with positive injectivity radius. We denote its intrinsic distance by $d$ and sectional curvature by $\K$.
\end{namedthm}


\section{Preliminaries on parallel transport and Lipschitz continuity}
\label{sect:prelim}

In Appendix \ref{app:terminology} we review some basic concepts and terminology from differential geometry that are relevant to the present work, and also introduce some notation. We invite the reader to look there whenever unfamiliar with some of the concepts or notation used in the main body of the paper. In particular, in Appendix \ref{app:terminology} we briefly discuss the logarithm map, normal and totally normal neighbourhoods, the injectivity radius, geodesics, convex sets, normal charts and the push-forward, as well as recall some of their properties important to our analysis; we also give a useful relationship between the Euclidean norm and the intrinsic distance on $\M$. 

Anywhere in the paper, for $x \in \M$ we write $\ap{\cdot,\cdot}_x$ and $\norm{\cdot}_x$ the tangent inner product and norm, respectively, on $T_x\M$, the tangent space of $\M$ at $x$. The tangent bundle of $\M$ is denoted by $TM$.

In this section, we present some background on Lipschitz continuity by parallel transport and its link to Hessians, as well as give the Cauchy--Lipschitz theorem for flow maps on manifolds.

\subsection{Lipschitz continuity via parallel transport}
\label{subsect:Lip-pp}
Given a curve $\gamma\:[0,1] \to \M$ and $v\in T_{\gamma(0)}\M$, the \emph{parallel transport} of $v$ along $\gamma$ is given by the unique solution $X\:[0,1] \to T\M$ with $X(t) \in T_{\gamma(t)} \M$ for all $t\in [0,1]$ to the ODE
\bes
	\begin{cases} \conn_{\gamma'(t)} X(t) = 0,\\ X(0) = v, \end{cases}
\ees
where $\conn_{\gamma'} X$ denotes the covariant derivative of $X$ along $\gamma$ and $\conn$ is the Levi-Civita connection on $\M$. We denote this solution by $\Pi_{\gamma,t}v$ for all $t\in[0,1]$. Given $t\in[0,1]$, the map $v \mapsto \Pi_{\gamma,t}v$ is a linear isometry from $T_{\gamma(0)}\M$ to $T_{\gamma(t)}\M$, i.e.,
\bes
	\ap{\Pi_{\gamma,t}v,\Pi_{\gamma,t}w}_{\gamma(t)} = \ap{v,w}_{\gamma(0)}, \qquad \mbox{for all $v,w\in T_{\gamma(0)}\M$}.
\ees

The situation of interest in this paper is when $\gamma\:[0,1] \to \M$ is the \emph{unique} minimizing geodesic connecting a point $x$ to a point $y$. In this case we write $\Pi_{xy}$ for $\Pi_{\gamma,1}$. It holds that $\Pi_{xy}^{-1} = \Pi_{yx}$, which means that taking the parallel transport of a vector $v\in T_x\M$ to $T_y\M$ and then back to $T_x\M$ returns the same vector $v$.

A \emph{vector field} on a set $U\subset\M$ is a function $X\:U \to T\M$ such that $X(x) \in T_x\M$ for all $x\in U$. Using parallel transport, we can give the following definition of Lipschitz continuous vector fields. We use in our setup a totally normal neighbourhood of $\M$, as any two points in such a set can be connected by a unique minimizing geodesic (though not necessarily lying entirely in it).
\begin{defn}[Lipschitz continuity via parallel transport]
\label{defn:Lip-parallel}
	Suppose that $U$ is a totally normal neighbourhood of $\M$ and let $X$ be a vector field on $U$. We say that $X$ is \emph{locally Lipschitz continuous} if for all compact sets $Q\subset U$ there exists a constant $L_Q>0$ such that
	\be\label{eqn:Lip-vf}
		\norm{X(x) - \Pi_{yx}X(y)}_x \leq L_Q\, d(x,y), \qquad \mbox{for all $x,y\in Q$};
	\ee
	we write $\norm{X}_{\mathrm{Lip}(Q)}$ for the smallest such constant. We say that $X$ is \emph{(globally) Lipschitz continuous} if there exists $L>0$ such that \eqref{eqn:Lip-vf} holds for all $x,y \in U$ by replacing $L_Q$ with $L$; we write $\norm{X}_{\mathrm{Lip}(U)}$ for the smallest such $L$.
\end{defn}
We note that a locally Lipschitz continuous vector field $X$ on a totally normal neighbourhood $U$ is also locally bounded in the sense that for all $Q\subset U$ compact there exists a constant $C_Q>0$ such that $\norm{X(x)}_x \leq C_Q$ for all $x\in Q$; we denote by $\norm{X}_{L^\infty(Q)}$ the smallest such constant. 

We also note that this definition of Lipschitz continuity is chart-free. The notion of Lipschitz continuity on charts is the standard concept used to show the well-posedness of flow maps via the Cauchy--Lipschitz theorem in the theory of dynamical systems on manifolds \cite{Irwin2001,FePaPa2020}. For completeness, we recall the definition here: a vector field $X$ on an open subset $U$ of $\M$ is \emph{locally Lipschitz continuous on charts} if for every chart $(V,\varphi)$ of $\M$ and compact set $Q\subset U \cap V$, there exists a constant $L_{\varphi,Q}>0$ such that
\be
\label{eqn:Lip-vf-charts}
\norm{\varphi_*X(x) - \varphi_*X(y)}_{\R^\dim} \leq L_{\varphi,Q}\, \norm{\varphi(x)-\varphi(y)}_{\R^\dim}, \qquad \mbox{for all $x,y\in Q$}.
\ee
The link between local Lipschitz continuity in the sense of Definition \ref{defn:Lip-parallel} and that on charts is given in the following lemma, which, in short, states that Lipschitz continuity in the sense of parallel transport implies Lipschitz continuity on \emph{normal} charts.

\begin{lem}
  \label{lem:pf-Lip} 
  Let $x\in\M$ and set $\delta\leq r_{\mt{conv}}(x)$. Suppose that $X$ is a vector field on the geodesic ball $B_\delta(x)$ which is locally Lipschitz continuous. Let moreover $(V,\varphi)$ be a normal chart generated by $x$. Then, for any $Q\subset B_\delta(x) \cap V$ compact there exists $L_{\varphi,Q}>0$ such that \eqref{eqn:Lip-vf-charts} holds; furthermore, the constant $L_{\varphi,Q}$ depends linearly on $\norm{X}_{L^\infty(Q)}$ and $\norm{X}_{\mt{Lip}(Q)}$.
\end{lem}
\begin{proof}
  See Appendix \ref{appendix:pf-Lip}.
\end{proof}

\subsection{Hessians}
\label{subsect:hess}

Given a vector field $X$ on an open subset $U$ of $\M$, a point $x\in U$ and a nonzero tangent vector $v\in T_x\M$, the \emph{derivative of $X$ at $x$ in direction $v$} is given, whenever the limit below exists, by 
\begin{equation}
\label{eqn:cov-der}
	\conn_v X(x) = \frac{\der}{\der t} \Bigl \vert_{t=0} \Pi_{x\gamma(t)}^{-1} X(\gamma(t)) = \lim_{t \to 0} \frac{\Pi_{x\gamma(t)}^{-1} X(\gamma(t))- X(x)}{t},
\end{equation}
where $\gamma$ denotes the geodesic starting at $x$ with velocity $v$ defined at any $t$ small enough for $\exp_x(tv)$ to exist. When indeed the limit above exists, we have $\conn_v X(x) \in T_x\M$ and we say that $X$ is \emph{differentiable at $x$ in direction $v$}. If $X$ is differentiable at $x$ in every direction, the map $\conn X(x) \: T_x\M \to T_x\M$ is linear and is called the \emph{derivative of $X$ at $x$}. If $X$ is differentiable at all points in $U$ in every direction, then we simply say that $X$ is \emph{differentiable} and the map $U\ni x \mapsto \conn X(x)$ is simply referred to as the \emph{derivative of $X$}.

For a differentiable function $f\:U \to \R$ such that the vector field $\grad f$ is differentiable at $x$ in direction $v$, we write $\Hess_vf(x)$ for the \emph{Hessian of $f$ at $x$ in direction $v$}, defined by
\begin{equation*}
  \Hess_vf(x) = \conn_v (\grad f)(x).
\end{equation*}
If $\grad f$ is differentiable at $x$ in every direction, the map $\Hess f(x)\: T_x\M \to T_x\M$ is linear and is referred to as the \emph{Hessian operator} of $f$ at $x$. If furthermore $\grad f$ is differentiable at all points in $U$ in every direction, then we simply say that $f$ has a \emph{Hessian} and the map $U\ni x \mapsto \Hess f(x)$ is called the \emph{Hessian of $f$}. Similar to the Euclidean case, the lemma below links bounded Hessians to Lipschitz continuity in the parallel transport sense \cite{FeLoPr2019}:
\begin{lem}
\label{lemma:bHessian-local}
Let $U\subset \M$ be a totally normal neighbourhood and let $f\:U \to \R$ have a Hessian. If the gradient of $f$ is locally Lipschitz continuous, then its Hessian is locally bounded, i.e., for all $Q\subset U$ compact there is $C_Q>0$ such that
\begin{equation*}
\norm{\Hess_v f(x)}_x \leq C_Q\norm{v}_x, \qquad \text{for all $x \in Q$ and $v\in T_x\M$}.
\end{equation*}
If furthermore $U$ is geodesically convex, then the converse is also true.
\end{lem}
\begin{proof}
See Appendix \ref{appendix:bHessian-local}.
\end{proof}

We directly have the following global version of Lemma \ref{lemma:bHessian-local}:
\begin{lem}
\label{lemma:bHessian}
Let $U\subset \M$ be a totally normal neighbourhood and let $f\:U \to \R$ have a Hessian. If the gradient of $f$ is Lipschitz continuous, then its Hessian is bounded, i.e., there is $C>0$ such that
\begin{equation}
\label{eqn:hess-bdd}
\norm{\Hess_v f(x)}_x \leq C\norm{v}_x, \qquad \text{for all $x \in U$ and $v\in T_x\M$}.
\end{equation}
If furthermore $U$ is geodesically convex, then the converse is also true.
\end{lem}

In the optimization literature, the property of a function $f$ defined on a totally normal neighbourhood $U\subset \M$ to have a Lipschitz continuous gradient,
is sometimes referred to as \emph{$L$-smoothness} \cite{AlOrBeLu2020,AlOrBeLu2020_2,ZhaSra2018}. Lemma \ref{lemma:bHessian} therefore says that $L$-smoothness and Hessian boundedness are equivalent on geodesically convex subsets of $\M$ (as they are in Euclidean spaces).

\subsection{Flow maps and the Cauchy--Lipschitz theorem}
\label{subsect:flow-maps}

A \emph{time-dependent vector field} on a subset $U\times [0,T)$ of $\M\times[0,\infty)$ is a function $X \: U \times [0,T) \to T\M$ such that $X(\cdot,t)$ is a vector field on $U$ for all $t\in [0,T)$; we will also use $X_t$ to denote $X(\cdot,t)$. Given such a time-dependent vector field $X$ and $\Sigma \subset U$ measurable, a \emph{flow map} generated by $(X,\Sigma)$ is a function $\Psi_X\: \Sigma \times [0,\tau) \to U$, for some $\tau\leq T$, that for all $x\in\Sigma$ and $t\in[0,\tau)$ satisfies
\begin{equation} \label{eqn:characteristics-general}
  \begin{cases} \dfrac{\der}{\der t} \Psi^t_X(x) = X_t(\Psi^t_X(x)),\\[10pt]
    \Psi^0_X(x) = x, \end{cases}
\end{equation}
where we used the abbreviation $\Psi^t_X$ for $\Psi_X(\cdot,t)$. A flow map is said to be \emph{maximal} if its time domain cannot be extended while \eqref{eqn:characteristics-general} holds; it is said to be \emph{global} if $\tau=T=\infty$ and \emph{local} otherwise.


The following theorem gives the local well-posedness of flow maps generated by Lipschitz vector fields in the sense of Definition \ref{defn:Lip-parallel}.
\begin{thm}[Local Cauchy--Lipschitz Theorem]\label{thm:Cauchy-Lip}
  Suppose that $U$ is a totally normal neighbourhood of $\M$. Let $T\in(0,\infty]$ and let $X$ be a time-dependent vector field on $U\times [0,T)$. Suppose that the vector fields in $\{X_t\}_{t\in [0,T)}$ are locally Lipschitz continuous and satisfy, for any compact sets $Q\subset U$ and $S\subset [0,T)$,
  \be\label{eqn:bdd-Lip}
  \int_S \left( \norm{X_t}_{L^\infty(Q)} + \norm{X_t}_{\mathrm{Lip}(Q)} \right) \d t < \infty.
  \ee
  Then, for every compact subset $\Sigma$ of $U$, there exists a unique maximal flow map generated by $(X,\Sigma)$.
\end{thm}
\begin{proof}
  The proof follows very closely that given in \cite{FePaPa2020} of the more classical version of the theorem for Lipschitz continuous vector fields on charts, with the additional use of Lemma \ref{lem:pf-Lip}. For the details, see Appendix \ref{appendix:Cauchy-Lip}.
\end{proof}


Theorem \ref{thm:Cauchy-Lip} gives \emph{local} well-posedness for flow maps. As given by the theorem below, we have \emph{global} well-posedness whenever $\M$ is convex and the vector fields in question are globally Lipschitz on $\M\times [0,\infty)$.

\begin{thm}[Global Cauchy--Lipschitz Theorem on Convex Manifold]
  \label{thm:Cauchy-Lip-glob}
  Assume that $\M$ is geodescially convex and let $X$ be a time-dependent vector field on $\M\times [0,\infty)$. Suppose that the vector fields in $\{X_t\}_{t\in [0,\infty)}$ are Lipschitz continuous and for any compact set $Q\subset \M$ there holds
  \bes
  \sup_{t\in[0,\infty)} \left(\norm{X_t}_{L^\infty(Q)} + \norm{X_t}_{\mathrm{Lip}(\M)}\right) < \infty.
  \ees
  Then, for every compact subset $\Sigma$ of $\M$, there exists a unique global flow map generated by $(X,\Sigma)$.
\end{thm}
\begin{proof}
  The proof follows from the Escape Lemma if $\M$ is compact and from a classical extension argument if $\M$ is unbounded. We refer the reader to Appendix \ref{appendix:Cauchy-Lip-glob} for the details.
\end{proof}



\section{Preliminaries on the interaction equation}
\label{sect:interaction-eq}


For simplicity, as we have already implicitly done in the previous section, we will drop the subindices $\M$ on the differential operators in equation \eqref{eqn:model}. We present in this section the notion of measure-valued solutions for \eqref{eqn:model} and the $1$-Wasserstein probability space. We only give the minimal tools we will need for our purpose; for a thorough theory of probability spaces and continuity equations in Euclidean space, we refer the reader to \cite{AGS2005}. 


\subsection{Notion of solution for the interaction equation} 
\label{subsect:solution}

For $U \subset \M$ open, denote by $\P(U)$ the set of Borel probability measures on the metric space $(U,d)$ and by $\Cont([0,T);\P(U))$ the set of continuous curves from $[0,T)$ into $\P(U)$ endowed with the narrow topology (i.e., the topology dual to the space of continuous bounded functions on $U$; see \cite{AGS2005}). 

If $\Psi\: \Sigma \to U$ for some measurable $\Sigma\subset U$, we denote by $\Psi \# \rho$ the \emph{push-forward} in the mass transport sense of $\rho$ through $\Psi$. Equivalently, $\Psi\#\rho$ is the probability measure such that for every measurable function $\zeta\: U \to [-\infty,\infty]$ with $\zeta\circ \Psi$ integrable with respect to $\rho$, we have
\bes
	\int_U \zeta(x) \d (\Psi \# \rho)(x) = \int_\Sigma \zeta(\Psi(x)) \d\rho(x).
\ees

For $T\in(0,\infty]$ and a curve $(\rho_t)_{t\in [0,T)} \subset \P(U)$, we denote by $\V[\rho]\: U \times [0,T) \to T\M$ the velocity vector field associated to \eqref{eqn:model}, that is,
\begin{equation} \label{eqn:v-field}
	\V[\rho] (x,t) =  -\grad K *\rho_t (x), \qquad \mbox{for all $(x,t) \in U \times [0,T)$},
\end{equation}
where for convenience we used $\rho_t$ in place of $\rho(t)$, as we shall often do in the sequel. Whether $v[\rho]$ is well-defined or not depends on $U$; indeed, we shall see that for $\grad K$ to make sense in the convolution with $\rho_t$, the set $U$ should be such that $\log_xy$ exists for all $x,y\in U$.

We define solutions to \eqref{eqn:model} in a geometric way, as the push-forward of the initial data through the corresponding flow map \cite[Chapter~8.1]{AGS2005}. Specifically, we adopt the following definition as solution to equation \eqref{eqn:model}:

\begin{defn}[Solution]
\label{defn:solution}
	Given $U\subset\M$ open, we say that $(\rho_t)_{t\in[0,T)} \subset \P(U)$ is a \emph{weak solution} to \eqref{eqn:model} if $(\V[\rho],\supp(\rho_0))$ generates a unique flow map $\Psi_{\V[\rho]}$ defined on $\supp(\rho_0)\times [0,T)$, and $\rho_t$ satisfies the implicit relation
\be \label{eqn:rho-push-forward}
	\rho_t = \Psi^t_{\V[\rho]} \# \rho_0, \qquad \mbox{for all $t\in[0,T)$}.
\ee
\end{defn}

It can be shown that solutions in the sense of Definition \ref{defn:solution} are also weak solutions in the sense of distributions to equation \eqref{eqn:model}; see \cite[Lemma 8.1.6]{AGS2005}.


\subsection{Wasserstein distance} 
\label{subsect:Wasserstein}

We will use the \emph{intrinsic $1$-Wasserstein distance} to compare solutions to \eqref{eqn:model}. Let $U\subset \M$ be open. For $\rho,\sigma \in \P(U)$, this distance is defined as:
\bes
	W_1(\rho,\sigma) = \inf_{\pi \in \Pi(\rho,\sigma)} \int_{U\times U} d(x,y) \d\pi(x,y),
\ees
where $\Pi(\rho,\sigma) \subset \P(U\times U)$ is the set of transport plans between $\rho$ and $\sigma$, i.e., the set of elements in $\P(U\times U)$ with first and second marginals $\rho$ and $\sigma$, respectively. 

Denote by $\P_1(U)$ the set of probability measures on $U$ with finite first moment and by $\P_\infty(U) \subset \P_1(U)$ the set of probability measures on $U$ with compact support. Both spaces $(\P_1(U),W_1)$ and $(\P_\infty(U),W_1)$) are well-defined metric spaces. In addition, we metrize the space $\Cont([0,T);\P_1(U))$ (and thus $\Cont([0,T);\P_\infty(U))$) with the distance $\bd_1$ defined by
\bes
	\bd_1(\rho,\sigma) = \sup_{t \in [0,T)} W_1(\rho_t,\sigma_t), \qquad \mbox{for all $\rho,\sigma \in \Cont([0,T);\P_1(U))$}.
\ees
Of course, when $U$ is bounded we have $\P(U) = \P_1(U) = \P_\infty(U)$.

The lemma below, used later in the paper, contains various Lipschitz properties involving the distance $W_1$.
\begin{lem}\label{lem:preliminary}
	Let $U\subset \M$ be open.
	\begin{enumerate}[label=(\roman*)]
		\item\label{it:prel1} Let $\Sigma\subset U$. Let furthermore $\rho \in \P_1(U)$ with $\supp(\rho) \subset \Sigma$ and $\Psi_1,\Psi_2\:\Sigma \to U$ be measurable functions. Then,
			\bes
				W_1({\Psi_1}\#\rho,{\Psi_2}\#\rho) \leq \sup_{x \in \supp(\rho)} d(\Psi_1(x),\Psi_2(x)).
			\ees
		\item\label{it:prel2} Let $T\in(0,\infty]$ and let $X$ be a time-dependent vector field on $U\times[0,T)$. Let $\rho \in \P_1(U)$ and suppose that $(X,\supp(\rho))$ generates a flow map $\Psi_X$ defined on $\supp(\rho)\times[0,\tau)$ for some $\tau\leq T$. Suppose furthermore that $X$ is bounded, i.e., there exists $C>0$ such that $\norm{X(x,t)}_{x\in U}<C$ for all $x\in U$ and $t\in[0,T)$. Then,
			\bes
				W_1({\Psi^t_X} \# \rho,{\Psi^s_X} \# \rho) \leq C|t-s|, \quad \quad \mbox{for all $t,s \in [0,\tau)$}.
			\ees
	\end{enumerate}
\end{lem}
\begin{proof}
We refer to \cite[Lemma 2.3]{FePaPa2020} for the proof of this result.
\end{proof}


\section{Well-posedness of the interaction equation}
\label{sect:local-wp}

We study the well-posedness of solutions to \eqref{eqn:model} with an interaction potential $K$ that satisfies:
\begin{namedthm}{(K)}
\namedlabel{hyp:K}{{\bfseries \upshape (K)}}
  $K\: \M \times \M \to \R$ has the form
  \begin{equation}
    \label{eqn:K-gen}
    K(x,y) = g(d(x,y)^2), \qquad \mbox{for all } x,y\in \M,
  \end{equation}
  where $g\: [0,\infty) \to \R$ is differentiable. 
\end{namedthm}

Because the interaction equation \eqref{eqn:model} involves the gradient of $K$, we are interested in the Lipschitz continuity of the gradient of the squared distance function $d$. For all $x,y\in \M$ define $d_y(x) = d(x,y)$; if $x$ and $y$ have a unique minimizing geodesic linking them, it holds that
\begin{equation}
\label{eqn:grad-d2}
\log_x y = - \tfrac{1}{2} \grad d^2_y(x).
\end{equation}
Similarly, for all $x,y\in\M$ write $K_y(x) = K(x,y)$, and again if there exists a unique minimizing geodesic linking $x$ and $y$ the chain rule yields
\begin{equation}
\label{eqn:gradK-gen}
	\grad K_y(x) = -2 g'(d(x,y)^2) \log_x y.
\end{equation}
In other words, equations \eqref{eqn:grad-d2} and \eqref{eqn:gradK-gen} only hold for points $y$ away from the cut locus of $x$. In particular, they hold for points $x,y\in\M$ with $d(x,y) < \iM$.


Equation \eqref{eqn:model} can be interpreted as an aggregation model using \eqref{eqn:v-field} and \eqref{eqn:gradK-gen}. Indeed, when a point mass at $x$ interacts with a point mass at $y$, the mass at $x$ is driven by a force of magnitude proportional to $|g'(d(x,y))^2|d(x,y)$, to move either towards $y$ (provided $g'(d(x,y)^2) >0$) or away from $y$  (provided $g'(d(x,y)^2) < 0$). The velocity field at $x$ computed by \eqref{eqn:v-field} accounts for all contributions from interactions with point masses $y \in U$ through the convolution. 


\subsection{Fundamental lemmas}
\label{subsect:fundamental-lemmas}

We give here several fundamental lemmas which will be at the core of our proof of the well-posedness of solutions to \eqref{eqn:model}.
Note first that, given $x,y\in\M$ and $U\subset \M$ a normal neighbourhood of both $x$ and $y$, there holds
\begin{equation}
  \label{eqn:parallel-log}
  \Pi_{yx} \log_y(x) = -\log_x(y).
\end{equation}
This is immediate from the fact that the parallel transport is an isometry. Indeed, since $\log_y(x)$ is tangent to the geodesic joining $y$ and $x$, the vector $\Pi_{yx} \log_y(x)$ is tangent to the geodesic at $x$ and has length $\| \log_y (x)\|_y = d(x,y)$.

We begin with a result on time-dependent vector fields. 

\begin{lem}[Fundamental Lemma I]\label{lem:dist-flow-maps}
Let $U$ be a totally normal neighbourhood of $\M$ and let $X,Y$ be two time-dependent vector fields on $U$. Let $\Sigma\subset U$ be measurable and suppose that $\Psi_X$ and $\Psi_Y$ are flow maps defined on $\Sigma\times [0,\tau)$, for some $\tau >0$, generated by $(X,\Sigma)$ and $(Y,\Sigma)$, respectively. Assume furthermore that $X$ is Lipschitz continuous with respect to its first variable, uniformly with respect to its second variable, i.e., there exists $L>0$ such that
	\bes
		\norm{X_t(x) - \Pi_{yx}X_t(y)}_x \leq L\, d(x,y), \quad \quad \mbox{for all $(x,y,t) \in U \times U \times [0,T)$}.
	\ees
	Then, for all $x\in\Sigma$, there holds
	\bes
		d(\PsiX(x),\PsiY(x)) \leq \frac{e^{Lt} -1}{L} \|X-Y\|_{L^\infty(U\times [0,\tau))},\quad \quad \mbox{for all $t\in[0,\tau)$}.
	\ees
\end{lem}

\begin{proof}
Let $x \in \Sigma$ and, for a better readability of the following calculations, write $p=\Psi_X^t(x)$ and $q = \Psi_Y^t(x)$; note that $p,q \in U$. We can assume $p\neq q$ or the result is trivial.

Since $U$ is totally normal, the function $t\mapsto d^2(p,q) = d^2(\Phi_X^t(x),\Phi_Y^t(x))$ is differentiable and, for all $t\in [0,\tau)$, we have
\be \label{eqn:est0}
\begin{split}
	\frac12 \frac{\der}{\der t} d^2(p,q) & = \frac12 \ap{ \grad d_{q}^2(p), X_t(p) }_{p} + \frac12 \ap{ \grad d_{p}^2(q), Y_t(q) }_{q}\\[2pt]
	&= - \ap{ \log_{p}(q), X_t(p) }_{p} - \ap{ \log_{q}(p), Y_t(q) }_{q}\\[2pt]
	&= - \ap{ \log_{p}(q), X_t(p) }_{p} - \ap{ \Pi_{qp} \log_{q}(p), \Pi_{qp} Y_t(q) }_p =: I + \mathit{II},
\end{split}
\ee
where for the third equality we used that the parallel transport is an isometry. 

We add and subtract the quantities
\[
	A: = \ap{ \Pi_{qp} \log_{q}(p), X_t(p) }_p \qquad \mbox{and} \qquad B:= \ap{ \Pi_{qp}\log_{q}(p), \Pi_{qp}X_t(q) }_p
\]
to the right-hand side of \eqref{eqn:est0}, which now reads
\[
	I - A + A - B + B + \mathit{II}. 
\]
The term $\mathit{II} + B$ estimates as
\begin{align*}
	\mathit{II} + B &= \ap{ \Pi_{qp}\log_{q}(p), \Pi_{qp}X_t(q) - \Pi_{qp}Y_t(q) }_p\\[2pt]
	&\leq \norm{ \Pi_{qp}\log_{q}(p) }_p \norm{ \Pi_{qp}(X_t-Y_t)(q) }_p\\[2pt]
	&= \norm{ \log_{q}(p) }_q \norm{ X_t(q)-Y_t(q) }_q\\[2pt]
	&\leq d(p,q) \norm{ X-Y }_{L^\infty(U\times[0,\tau))},
\end{align*}
where we used again that $\Pi_{qp}$ is an isometry and that $\norm{\log_q(p)}_q = d(p,q)$. Also estimate
\begin{align*}
	A - B &= \ap{ \Pi_{qp}\log_{q}(p),  X_t(p) - \Pi_{qp}X_t(q)}_p\\[2pt]
	&\leq \norm{ \Pi_{qp}\log_{q}(p) }_p \norm{ X_t(p) - \Pi_{qp}X_t(q) }_p\\[2pt]
	&\leq \norm{ \log_{q}(p) }_q L\, d(p,q)\\[2pt]
	&= L\, d(p,q)^2,
\end{align*}
where we used the Lipschitz continuity of $X$, the isometric property of $\Pi_{qp}$ and that $\norm{\log_q(p)}_q = d(p,q)$. Finally, by \eqref{eqn:parallel-log}, we yield  
\bes
	I - A = - \ap{ \log_{p}(q) + \Pi_{qp}\log_{q}(p), X_t(p) }_p = 0.
\ees

Using these estimates in \eqref{eqn:est0}, we find
\[
	\frac12 \frac{\der}{\der t} d^2(p,q) \leq L\, d^2(p,q) + \norm{X-Y}_{L^\infty(U\times[0,\tau))} d(p,q),
\]
and after cancelling a $d(p,q)$, we get:
\[
	\frac{\der}{\der t} d(p,q) \leq L\, d(p,q) + \norm{X-Y}_{L^\infty(U\times[0,\tau))}.
\]
Gronwall's lemma now gives the desired result.
\end{proof}

We now focus on what restrictions we should further impose on a neighbourhood $U$ for the well-posedness to hold on $U$. Given a point $z\in\M$, \cite[Theorem 6.6.1]{Jost2017} gives a bound on the Hessian of $d_z^2$ provided the sectional curvature $\K$ of $\M$ is locally bounded. Specifically, let $r<\iM$ so that the exponential map $\exp_z$ is a diffeomorphism on $B_r(0)\subset T_z\M$. Denote by $B_r(z) := \exp_z(B_r(0))$ the geodesic ball centred at $z$ with radius $r$, and suppose that there exist reals $\lambda\leq 0$ and $0\leq\mu\leq\frac{\pi^2}{4r^2}$ such that 
\begin{equation}
\label{eqn:K-bounds}
\lambda \leq \K \leq \mu, \qquad \text{ on $B_r(z)$}.
\end{equation}
Then, by \cite[Theorem 6.6.1]{Jost2017} it holds that
\begin{equation}
\label{eqn:bHessian}
2\sqrt{\mu}\, d(x,z) \cot( \sqrt{\mu}\, d(x,z)) \|v\|_x^2 \leq \ap{\Hess_v d_z^2(x),v}_x \leq 2 \sqrt{-\lambda}\, d(x,z) \coth( \sqrt{-\lambda}\, d(x,z)) \|v\|_x^2,
\end{equation}
for all $x \in B_r(z)$ and $v \in T_x \M$. Note that the assumption of completeness and simple connectedness of $\M$ in \ref{hyp:M} ensures that $\iM = \infty$ whenever $\K\leq0$ everywhere (cf. \cite[Corollary 6.9.1]{Jost2017}, a consequence of the Cartan--Hadamard theorem). 

Let us further introduce some notation used throughout the rest of the paper:
\begin{notn}
  \label{notn:local-curv}
  Because $\M$ is locally compact (since finite-dimensional) and $\K$ is continuous, $\K$ is also locally bounded (cf. \cite[Remark 2.2]{Mihai1976}). For any $p\in\M$ and $r>0$, we write $\lambda_r(p)\leq0$ and $\mu_r(p)\geq0$ any lower and upper bounds for $\K$ on the set $\{x\in\M \st d(p,x) < r \}$ (which, by the terminology set up in Appendix \ref{app:terminology}, is denoted by $B_r(p)$ if $r\leq r_\mathrm{inj}(p)$). 
\end{notn}


\begin{lem}[Fundamental Lemma II]
  \label{lemma:d2-Hessian}
  Let $U\subset B_{r/2}(p)$ be open for some $p\in\M$ and $r < \iM$, and denote $\lambda = \lambda_{2r}(p)$ and $\mu = \mu_{2r}(p)$. Suppose moreover that $U$ satisfies
  \begin{equation}
    \label{eqn:diam-U}
    \diam:=\Diam(U) < \frac{\pi}{2 \sqrt{\mu}}, \qquad \text{if $\mu>0$}.
  \end{equation}
  Then, for all $z\in U$, the Hessian $\Hess d^2_z$ is bounded on $U$ by $L:=2\sqrt{-\lambda}\, \diam \coth( \sqrt{-\lambda}\, \diam)$.

If moreover $U$ is geodesically convex, we have
\begin{equation}
\label{eqn:d2-Lipschitz}
\| \log_x z - \Pi_{yx} \log_y z \|_x \leq \tfrac{L}{2} d(x,y), \qquad \text{ for all } x,y,z \in U.
\end{equation}
\end{lem}
\begin{proof}
  The proof is a direct consequence of \eqref{eqn:bHessian}. See Appendix \ref{appendix:d2-Hessian} for the details.
\end{proof}

\begin{rem}
  \label{rem:Euclidean-Lip}
  When $\lambda_r(p)=\mu_r(p)=0$ for all $p\in \M$ and $r>0$, we recover the trivial result that $L=2$ in Lemma \ref{lemma:d2-Hessian} (and, in particular, $L$ is independent of the diameter of the subset $U$), and \eqref{eqn:d2-Lipschitz} holds for all $x,y,z\in \M$.
\end{rem}

Lemma \ref{lemma:d2-Hessian} states that, for all $z$, the map $x\mapsto \log_xz$ is Lipschitz continuous in the sense of parallel transport. On the other hand, the following lemma shows that, for all $x$, the map $z\mapsto \log_xz$ is Lipschitz continuous in the classical Euclidean sense.

\begin{lem}[Fundamental Lemma III]
  \label{lemma:Lipschitz-2}
  Let $U\subset B_r(p)$ be a totally normal neighbourhood for some $p\in\M$ and $r < \iM$, and write $\mu = \mu_r(p)$. We distinguish two cases:
  \begin{enumerate}[label=(\roman*)]
    \item \ul{$\mu =0$.} Then,
    \begin{equation}
      \label{eqn:log-diff-mu0}
      \| \log_x y - \log_x z \|_x \leq d(y,z), \qquad  \text { for all } x,y,z \in U.
    \end{equation}
    \item \ul{$\mu >0$.} Take any $0<\e<\pi$ and further assume $\Diam(U) \leq \frac{\pi - \e}{\sqrt{\mu}}$. Then,
    \begin{equation}
      \label{eqn:log-diff-mu}
      \| \log_x z - \log_x y \|_x \leq  \frac{\pi -\e}{\sin(\pi - \e)} d(y,z), \qquad \text{ for all } x,y,z \in U.
    \end{equation}
  \end{enumerate}
\end{lem}
\begin{proof}
  This follows from \cite[Corollary 6.6.1]{Jost2017}. We give the details in Appendix \ref{appendix:Lipschitz-2}.
\end{proof}

\begin{rem}
  \label{rem:negative-Lip}
  When $\iM = \infty$ and $\mu_r(p)=0$ for all $p\in \M$ and $r>0$, i.e., $\M$ is globally nonpositively curved, \eqref{eqn:log-diff-mu0} holds for all $x,y,z\in \M$.
\end{rem}
\begin{rem}
  \label{rem:glob-neg-curv}
  If $\M$ is globally nonpositively curved, then the restriction on $U$ in Lemma \ref{lemma:d2-Hessian} in order to get a bounded Hessian simplifies into open and bounded. In Lemma \ref{lemma:Lipschitz-2} the restriction on $U$ simplifies into totally normal neighborhood and bounded.
\end{rem}


\subsection{Local well-posedness}
\label{subsect:fixed-point}

Let us now turn to the local well-posedness of solutions to Equation \eqref{eqn:model}. We consider an additional assumption:
\begin{namedthm}{(Kloc)}
  \namedlabel{hyp:Kloc}{{\bfseries \upshape (Kloc)}} 
  $K\: \M \times \M \to \R$ satisfies \ref{hyp:K} with $g'$ locally Lipschitz continuous. 
\end{namedthm}
\noindent For $\Delta >0$, we write $C_{g'}(\Delta)$ and $L_{g'}(\Delta)$ for the $L^\infty$ norm and the Lipschitz constant of $g'$ on $[0,\Delta^2]$.

\begin{lem}\label{lem:v-Lip}
  Assume the manifold $\M$ and potential $K$ satisfy \ref{hyp:M} and \ref{hyp:Kloc}, and take $U\subset B_{r/2}(p)$ open and geodesically convex for some $p\in\M$ and $r<\iM$ such that \eqref{eqn:diam-U} holds (cf. also Notation \ref{notn:local-curv}). Let $T\in (0,\infty]$ and $\rho\in \Cont([0,T);\P(U))$. Then the time-dependent vector field $v[\rho]$ given by \eqref{eqn:v-field} is bounded and Lipschitz continuous with respect to its first variable, uniformly with respect to its second variable. 
\end{lem}
\begin{proof}
Write $\Delta$ for $\Diam(U)$. We have, for all $x,y\in U$, 
\begin{align*}
	\norm{\grad K_y(x)}_x &= \norm{g'(d(x,y)^2) \grad d_y^2(x)}_x \leq 2 C_{g'}(\Delta) \Delta,
\end{align*}
where we used the local bound on $g'$ and that $\norm{ \grad d_y^2(x)}_x = 2 d(x,y) \leq 2 \Delta$. Immediately, for all $t\in[0,T)$, it follows
\be\label{eqn:v-bound}
	\| \V[\rho](x,t) \|_x \leq \int_{U} \| \grad K_y(x)\|_x  \d\rho_t(y) \leq 2 C_{g'}(\Delta) \Delta,
\ee
which shows that $\V[\rho]$ is bounded.

Also, for all $x,y,z\in U$, we get
\begin{align*}
	&\norm{\grad K_z (x) - \Pi_{yx} \grad K_z (y))}_x = \norm{g'(d(x,z)^2) \grad  d_z^2(x) - g'(d(y,z)^2) \Pi_{yx} \grad d_z^2(y) }_x\\[3pt]
	&\quad\leq |g'(d(x,z)^2)| \norm{\grad d_z^2(x) - \Pi_{yx}\grad d_z^2(y)}_x + \norm{\grad d_z^2(y)}_y |g'(d(x,z)^2) - g'(d(y,z)^2)|\\[3pt]
	&\quad\leq C_{g'}(\Delta) L\, d(x,y) + 2 d(y,z) L_{g'}(\Delta)  |d(x,z)^2 - d(y,z)^2|,
\end{align*}
where we used the Lipschitz continuity of the vector field $\grad d_z^2$ (equivalently of the logarithm map) given by \eqref{eqn:d2-Lipschitz} and the local bound and Lipschitz continuity of $g'$. Now use $d(y,z) \leq \Delta$ and $$|d(x,z)^2 - d(y,z)^2| = |d(x,z) - d(y,z)| |d(x,z) + d(y,z)| \leq 2 \Delta d (x,y),$$
to get
\[
\norm{\grad K_z (x) - \Pi_{yx} \grad K_z (y))}_x \leq \overbar{L}\, d(x,y) \qquad \text{ for all }x,y,z\in U,
\]
where
\begin{equation}
\label{eqn:Lbar}
\overbar{L} = C_{g'}(\Delta) L + 4 \Delta^2  L_{g'}(\Delta).
\end{equation}
Then, for all $t \in [0,T)$ and $x,y \in U$, we get
\begin{align}
\label{eqn:v-Lip}
	\norm{v[\rho](x,t) - \Pi_{yx} v[\rho](y,t)}_x &\leq \int_U \norm{\grad K_z (x) - \Pi_{yx} \grad K_z (y))}_x \d\rho_t(z) \nonumber \\[3pt]
	&\leq \overbar{L}\, d(x,y),
\end{align}
which shows that $v[\rho]$ is Lipschitz continuous with respect to its first variable (space), uniformly with respect to its second (time).
\end{proof}

\begin{rem}
\label{rem:unif-bounds}
The $L^\infty$ bound and the Lipschitz constant of $\V[\rho]$ in \eqref{eqn:v-bound} and \eqref{eqn:v-Lip} do not depend on the curve $\rho$. This is important for the proof of Theorem \ref{thm:well-posedness}. 
\end{rem}

\begin{lem}\label{lem:grad-lip-2}
  Let $\M$ and $K$ satisfy \ref{hyp:M} and \ref{hyp:Kloc} and $U\subset B_r(p)$ be a totally normal neighborhood for some $p\in\M$ and $r<\iM$, and write $\mu = \mu_r(p)$ (recall Notation \ref{notn:local-curv}). If $\mu > 0$ assume that $\Delta := \Diam(U)$ satisfies $\Delta \leq \frac{\pi - \e}{\sqrt{\mu}}$ for some  $0<\e<\pi$. Let $\rho,\sigma \in \Cont([0,T);\P(U))$. Then, there exists $\Lip>0$ so that
	\begin{equation*}
		\|\V[\rho]-\V[\sigma]\|_{L^\infty(U \times [0,T))} \leq \Lip \, \bd_1(\rho,\sigma).
	\end{equation*}
\end{lem}
\begin{proof}
By Lemma \ref{lemma:Lipschitz-2}, we have
\[
\| \log_x z - \log_x y \|_x \leq  \ell d(y,z) \qquad \text{ for all } x,y,z \in U,
\]
where $\ell =  \frac{\pi -\e}{\sin(\pi - \e)}$ if $\mu >0$ and $\ell = 1$ if $\mu =0$. Hence, using \eqref{eqn:grad-d2} we find
\[
\| \nabla d^2_y(x) - \nabla d^2_z(x) \|_x \leq 2 \ell d(y,z), \qquad \text{ for all } x,y,z \in U.
\]

Since \ref{hyp:Kloc} holds, for all $x,y\in U$ compute
	\begin{align*}
		& \|\nabla K_y(x) - \nabla K_z(x)\|_x = \|g'(d(x,y)^2) \nabla d^2_y(x) - g'(d(x,z)^2) \nabla d^2_z(x)\|_x \nonumber \\[3pt]
		&\qquad \leq |g'(d(x,z)^2)| \|\nabla d^2_y(x) - \nabla d^2_z(x)\|_x + \|\nabla d^2_y(x)\|_x|g'(d(x,y)^2) - g'(d(x,z)^2)|,	
	\end{align*}
where we added and subtracted $g'(d(x,z)^2) \nabla d^2_y(x)$ on the first line and used the triangle inequality. Then, using the bound and Lipschitz constant of $g'$, the fact that $\|\nabla d^2_y(x)\|_x= 2 d(x,y)$ and again the triangle inequality, we find
\begin{align}
 \|\nabla K_y(x) - \nabla K_z(x)\|_x & \leq 2 C_{g'}(\Delta)\ell\, d(y,z) + 2L_{g'}(\Delta) |d(x,y)+d(x,z)| |d(x,y)-d(x,z)| d(x,y) \nonumber \\[2pt]
& \leq (2 C_{g'}(\Delta)\ell + 4 L_{g'}(\Delta) \Delta^2) d(y,z).
	\label{eqn:ineq2-K}	
\end{align}

Now, for $(x,t) \in U \times [0,T)$, take $\pi_t \in \Pi(\rho_t,\sigma_t)$ to be an optimal transport plan between $\rho_t$ and $\sigma_t$, and estimate
	\begin{align*}
		\norm{\V[\rho](x,t) - \V[\sigma](x,t)}_x &= \left \| \int_{U} \nabla K_y(x)\d\rho_t(y) - \int_{U} \nabla K_z(x) \d\sigma_t(z) \right \|_x \\[3pt]
		&= \left \| \int_{U \times U} \nabla K_y(x) \d\pi_t(y,z) - \int_{U \times U} \nabla K_z(x) \d\pi_t(y,z) \right \|_x\\[3pt]
		&\leq \int_{U \times U} \norm{\nabla K_y(x) - \nabla K_z(x)}_x \d\pi_t(y,z).
	\end{align*}
Then, using \eqref{eqn:ineq2-K} we find
	\begin{align}
		\norm{\V[\rho](x,t) - \V[\sigma](x,t)}_x &\leq \Lip  \int_{U \times U} d(y,z) \d\pi_t(y,z) 
= \Lip\, W_1(\rho_t,\sigma_t) \nonumber \\[2pt]
	&\leq \Lip \bd_1(\rho,\sigma), \label{est:LipX}
	\end{align}
where
\begin{equation}
\label{eqn:Lip-Lambda}
\Lip = 2 C_{g'}(\Delta)\ell + 4 L_{g'}(\Delta) \Delta^2.
\end{equation}
Taking the supremum in $(x,t)\in U \times [0,T)$ on the left-hand side of \eqref{est:LipX} gives the result.
\end{proof}

\begin{rem}
  The assumptions on $U$ in Lemma \ref{lem:grad-lip-2} are weaker than those in Lemma \ref{lem:v-Lip}.  
\end{rem}

\begin{thm}[Local Well-Posedness]\label{thm:well-posedness}
	Assume the manifold $\M$ and interaction potential $K$ satisfy \ref{hyp:M} and \ref{hyp:Kloc}, and take $U\subset B_{r/2}(p)$ open and geodesically convex for some $p\in\M$ and $r<\iM$ so that \eqref{eqn:diam-U} holds (cf. Notation \ref{notn:local-curv}). Let $\rho_0 \in \P(U)$. Then, there exist $T>0$ and a unique weak solution in $\Cont([0,T);\P(U))$ starting from $\rho_0$ to the aggregation equation \eqref{eqn:model}.
	\end{thm}
\begin{proof} 
  By Lemma \ref{lem:v-Lip} the velocity field $\V[\sigma]$ (for any fixed continuous curve $\sigma$ in $\P(U)$) satisfies the assumptions of Theorem \ref{thm:Cauchy-Lip}, so that we can infer that $(\V[\sigma],\supp(\rho_0))$ generates a unique maximal flow $\Psi_{v[\sigma]}$ defined on $\supp(\rho_0)\times [0,\tau)$ for some $\tau\in(0,\infty]$. In addition, since the $L^\infty$ bound and the Lipschitz constant in \eqref{eqn:v-bound} and \eqref{eqn:v-Lip} do not depend on the underlying curve (see Remark \ref{rem:unif-bounds}), the maximal time of existence $\tau$ of $\Psi_{v[\sigma]}$ does not depend on $\sigma$ by the Cauchy--Lipschitz theorem (cf. Theorem \ref{thm:Cauchy-Lip}).

  We can then define the map $\Gamma$ by
  \begin{equation}
    \label{eqn:Gamma}
    \Gamma(\sigma)(t) = {\Psi_{\V[\sigma]}^t}\# \rho_0, \qquad \mbox{for all $\sigma \in \Cont([0,\tau);\P(U))$ and $t \in [0,\tau)$}.
  \end{equation}
We will prove that $\Gamma$ defines a map from $\Cont([0,\tau);\P(U))$ into itself and that it has a unique fixed point. This fixed point is the weak solution of \eqref{eqn:model} starting at $\rho_0$.

Note first that $\Gamma$ does indeed map $\Cont([0,\tau);\P(U))$ into itself. For $\sigma \in \Cont([0,\tau);\P(U))$ fixed, by definition of a flow map generated by $(\V[\sigma],\supp(\rho_0))$, we infer that $\Psi_{v[\sigma]}^t(x)\in U$ for all $x\in\supp(\rho_0)$ and $t\in[0,\tau)$. Consequently, $\Gamma(\sigma)(t)$ is supported in $U$ and moreover, it is a probability measure by conservation of mass through the push-forward, so that $\Gamma(\sigma)(t) \in\P(U)$, for all $t\in[0,\tau)$.  Furthermore, the map $t \mapsto \Gamma(\sigma)(t)$ is continuous due to the combination of Lemma \ref{lem:preliminary}\ref{it:prel2} with Lemma \ref{lem:v-Lip}. We conclude that $\Gamma \: (\Cont([0,\tau);\P(U)),\bd_1) \to (\Cont([0,\tau);\P(U)),\bd_1)$.
	
To show that $\Gamma$ is a contraction we will have to restrict the final time to some $T\leq \tau$ to be determined. Let $\rho, \sigma \in \Cont([0,\tau);\P(U))$. Then, for all $t \in [0,\tau)$, we have
	\begin{align}
		W_1({\Psi_{\V[\rho]}^t}\#\rho_0,{\Psi_{\V[\sigma]}^t}\#\rho_0) &\leq \sup_{x \in\supp(\rho_0)} d(\Psi_{\V[\rho]}^t(x),\Psi_{\V[\sigma]}^t(x)) \nonumber \\[3pt]
		&\leq C(t) \| \V[\rho] - \V[\sigma]\|_{L^\infty(U\times [0,\tau))} \nonumber \\[5pt]
		&  \leq C(t) \Lip \bd_1(\rho,\sigma),
		\label{eqn:estW1}
	\end{align}
	where for the first inequality we used Lemma \ref{lem:preliminary}\ref{it:prel1}, for the second inequality we used Lemmas \ref{lem:v-Lip} and \ref{lem:dist-flow-maps}, with
	\bes
		C(t) =  \frac{e^{\bar{L}t} -1}{\bar{L}},
	\ees
and for the last inequality we used Lemma \ref{lem:grad-lip-2}; the Lipschitz constants $\bar{L}$ and $\Lambda$ depend on $\Diam(U)$ and are defined in \eqref{eqn:Lbar} and \eqref{eqn:Lip-Lambda}. Since $t\mapsto C(t)$ is increasing, with $\lim_{t \to 0} C(t) = 0$ and $\Lip$ is independent of time, we can choose $T\leq \tau$ (only depending on $\bar L$ and $\Lambda$) small enough so that 
\[
C(t) \Lip < C(T) \Lip < \overbar{C}, \qquad \text{ for all } t \in [0,T),
\]
for some constant $\overbar{C}<1$. Restricting $T$ accordingly, by taking the supremum over $[0,T)$ in \eqref{eqn:estW1} we find
	\bes
		\bd_1(\Gamma(\rho),\Gamma(\sigma)) \leq \overbar{C} \bd_1(\rho,\sigma),
	\ees
with $\overbar{C}<1$. This shows that the restriction of $\Gamma$ to $(\Cont([0,T);\P(U)),\bd_1)$ is a contraction.

We have thus shown that $\Gamma\: (\Cont([0,T);\P(U)),\bd_1) \to (\Cont([0,T);\P(U)),\bd_1)$ has a unique fixed point, that is, there exists a unique $\rho \in \Cont([0,T);\P(U))$ such that
	\bes
		\rho_t = \Psi_{\V[\rho]}^t \# \rho_0 \quad \mbox{for all $t\in[0,T)$};
	\ees
this fixed point $\rho$ is the desired solution.
\end{proof}

\begin{rem}
  In Theorem \ref{thm:well-posedness}, one could choose $U = B_{r/2}(p)$ as long as this geodesic ball is convex and $r<\iM$, that is, as long as $r<\min(\iM,2\, \cM) = 2\,\cM$ (where the equality comes from the fact that $2\cM \leq \iM$ \cite[Proposition IX.6.1]{Chavel2006}).
\end{rem}


\subsection{Global well-posedness when $\M$ is nonpositively curved}
\label{subsect:global-wp}

We establish here the global well-posedness for \eqref{eqn:model} when $\M$ is nonpositively curved. Recall that in this case $\iM = \infty$; in fact, we have that the exponential map is a diffeomorphism from $T_x\M$ to $\M$ for all $x\in \M$ (see \cite[Corollary 6.9.1]{Jost2017}). In particular, this implies that the velocity field \eqref{eqn:v-field} is well-defined for all $\rho\in \Cont([0,\infty);\P_\infty(\M)$) and $\M$ is geodesically convex. 

We shall focus on the case when $\M$ is nonpositively curved with bounded curvature. We thus consider the following hypothesis, for some $\lambda\leq0$:
\begin{namedthm}{(M)$_\lambda$}
  \namedlabel{hyp:M}{{\bfseries \upshape (M)}} 
  $\M$ satisfies \ref{hyp:M} and its curvature $\K$ is so that $\lambda \leq \K \leq 0$ everywhere on $\M$. 
\end{namedthm}
\noindent We wish to adapt Lemmas \ref{lem:v-Lip} and \ref{lem:grad-lip-2} globally when $\M$ satisfies \ref{hyp:M}$_\lambda$ for some $\lambda\leq0$. If $\M$ is compact, then this is straightforward. Indeed, in this setting Lemmas \ref{lem:v-Lip} and \ref{lem:grad-lip-2} apply directly when replacing $U$ with $\M$ and $T$ with $\infty$ (cf. also Remark \ref{rem:glob-neg-curv}), and the following lemmas hold:

\begin{lem}\label{lem:v-Lip-glob-compact}
  Assume there exists $\lambda\leq0$ such that $\M$ satisfies \ref{hyp:M}$_\lambda$. Suppose $\M$ is compact and $K$ satisfies \ref{hyp:Kloc}. Let $\rho\in \Cont([0,\infty);\P(\M))$. Then the time-dependent vector field $v[\rho]$ given by \eqref{eqn:v-field} is bounded and Lipschitz continuous with respect to its first variable, uniformly with respect to its second variable. 
\end{lem}

\begin{lem}\label{lem:grad-lip-2-glob-compact}
  Assume $\K\leq0$. Suppose $\M$ satisfies \ref{hyp:M} and is compact and suppose $K$ satisfies \ref{hyp:Kloc}. Let furthermore $\rho,\sigma\in \Cont([0,\infty);\P(\M))$. Then, there exists $\Lip>0$ so that
  \begin{equation*}
    \|\V[\rho]-\V[\sigma]\|_{L^\infty(\M \times [0,\infty))} \leq \Lip \, \bd_1(\rho,\sigma).
  \end{equation*}
\end{lem}

When $\M$ is unbounded, Lemmas \ref{lem:v-Lip} and \ref{lem:grad-lip-2} do not adapt this easily. Consider in this case the following hypothesis on $K$, for some $\lambda\leq0$:

\begin{namedthm}{(Kglob)$_\lambda$}
  \namedlabel{hyp:Kglob}{{\bfseries \upshape (Kglob)}} 
  $K\: \M \times \M \to \R$ satisfies \ref{hyp:K} with $g$ such that, for some constant $A_{g'}>0$,
  \begin{equation*}
    \abs{g'(r^2)r - g'(s^2)s} \leq A_{g'} \abs{r-s}, \qquad \text{for all $r,s\geq0$}.
  \end{equation*}
  When $\lambda<0$, it further holds that
  \begin{equation*}
    \abs{g'(r^2)}r \leq A_{g'}, \qquad \text{for all $r\geq0$}.
  \end{equation*}
\end{namedthm}
\noindent With this notation, \ref{hyp:Kglob}$_0$ refers to the first bound only in \ref{hyp:Kglob}$_\lambda$.

The first bound in \ref{hyp:Kglob}$_\lambda$ means that the function $r\mapsto g'(r^2)r$ is globally Lipschitz continuous; it implies in particular that $g'$ is globally bounded by $A_{g'}$. The second bound in \ref{hyp:Kglob}$_\lambda$, in case $\lambda<0$, means that $r\mapsto g'(r^2)r$ is globally bounded; it also says that $g'(r)$ must decrease at least as fast as $1/\sqrt{r}$ as $r\to\infty$.

The Lipschitz condition in \ref{hyp:Kglob}$_\lambda$ is consistent with the classical Cauchy--Lipschitz theory in Euclidean space. Indeed, in this case this Lipschitz condition is enough to get the global Lipschitz continuity of the velocity field (and the local boundedness that follows from it) in order to obtain global well-posedness of the flow maps (cf. Theorem \ref{thm:Cauchy-Lip-glob}). When the space is negatively curved, the Lipschitz condition in \ref{hyp:Kglob}$_\lambda$ is not sufficient anymore to get the global Lipschitz continuity of the velocity field. This stems from \eqref{eqn:d2-Lipschitz}, where the Lispchitz constant depends on the distance between the points considered (as opposed to the flat case discussed in Remark \ref{rem:Euclidean-Lip})---the global upper bound condition in \ref{hyp:Kglob}$_\lambda$ helps counterbalance this effect. All this can be seen in details in the proof of Lemma \ref{lem:v-Lip-glob} below.

\begin{rem}
  \label{rem:glob-Lip-g}
  One can check that the Lipschitz condition in \ref{hyp:Kglob}$_\lambda$ implies
  \begin{equation*}
    \abs{g'(r^2) - g'(s^2)}s \leq 2A_{g'} \abs{r-s}, \qquad \text{for all $r,s\geq0$}.
  \end{equation*}
\end{rem}

\begin{lem}\label{lem:v-Lip-glob}
  Assume there exists $\lambda\leq0$ such that $\M$ satisfies \ref{hyp:M}$_\lambda$. Suppose $\M$ is unbounded and $K$ satisfies \ref{hyp:Kglob}$_\lambda$. Let $\rho\in \Cont([0,\infty);\P_\infty(\M))$. Then the time-dependent vector field $v[\rho]$ given by \eqref{eqn:v-field} is locally bounded and globally Lipschitz continuous with respect to its first variable, uniformly with respect to its second variable.
\end{lem}
\begin{proof}
We follow closely the proof of Lemma \ref{lem:v-Lip}. Let $U\subset\M$ be bounded, open and convex with $\Delta = \Diam(U)$. We have, for all $x,y\in U$, 
\begin{align*}
  \norm{\grad K_y(x)}_x &= \norm{g'(d(x,y)^2) \grad d_y^2(x)}_x \leq 2 A_{g'} \Delta,
\end{align*}
which yields the local boundedness of $\V[\rho](\cdot,t)$, uniformly with respect to $t\in[0,\infty)$. (Note that in case $\lambda<0$, this bound on $\V[\rho](\cdot,t)$ is actually global thanks to the second condition in \ref{hyp:Kglob}$_\lambda$.)

For all $x,y,z\in U$, we get
\begin{align*}
	&\norm{\grad K_z (x) - \Pi_{yx} \grad K_z (y))}_x = \norm{g'(d(x,z)^2) \grad  d_z^2(x) - g'(d(y,z)^2) \Pi_{yx} \grad d_z^2(y) }_x\\[3pt]
        &\quad \leq |g'(d(x,z)^2)| \norm{\grad d_z^2(x) - \Pi_{yx}\grad d_z^2(y)}_x + \norm{\grad d_z^2(y)}_y |g'(d(x,z)^2) - g'(d(y,z)^2)|\\[3pt]
        &\quad\leq |g'(\Delta^2)| L(\Delta)\, d(x,y) + 4A_{g'} \abs{d(x,z) - d(y,z)}\\[3pt]
        &\quad\leq (|g'(\Delta^2)| L(\Delta) + 4A_{g'})\, d(x,y),
\end{align*}
where for the first inequality we added and subtracted $g'(d(x,z)^2)  \Pi_{yx} \grad d_z^2(y)$ and used the triangle inequality; we then used the Lipschitz continuity of the vector field $\grad d_z^2$ given by \eqref{eqn:d2-Lipschitz} with $L(\Delta):=2\sqrt{-\lambda} \Delta\coth(\sqrt{-\lambda}\Delta)$, the fact that $\norm{\grad d_z^2(y)}_y = 2d(y,z)$ and Remark \ref{rem:glob-Lip-g} for the second inequality, and the reverse triangle inequality on $d$ for the third inequality. We notice that if $\lambda=0$, then $L(\Delta) = 2$ and so $|g'(\Delta^2)|L(\Delta)\leq 2A_{g'}$. If $\lambda<0$, we have
\begin{align*}
  g'(\Delta^2) L(\Delta) &= 2g'(\Delta^2)\sqrt{-\lambda}\Delta\coth(\sqrt{-\lambda}\Delta)\\[2pt]
  &\leq 
  \begin{cases}
    2g'(\Delta^2)\coth(1) & \text{if $\Delta \leq \frac{1}{\sqrt{-\lambda}}$}\\[2pt]
    2g'(\Delta^2) \sqrt{-\lambda}\Delta \coth(1) & \text{if $\Delta > \frac{1}{\sqrt{-\lambda}}$}
  \end{cases}\\[2pt]
  &\leq 
    \begin{cases}
    2 A_{g'}\coth(1) & \text{if $\Delta \leq \frac{1}{\sqrt{-\lambda}}$}\\[2pt]
    2 A_{g'}\sqrt{-\lambda}\coth(1) & \text{if $\Delta > \frac{1}{\sqrt{-\lambda}}$}
  \end{cases}\\
  &\leq 2 A_{g'} \coth(1)\max(1,\sqrt{-\lambda}),
\end{align*}
where for the first inequality we used that $r\coth(r)$ and $\coth(r)$ are bounded by $\coth(1)$, respectively for $r\leq 1$ and for $r>1$, and for the second inequality we used the second bound condition in \ref{hyp:Kglob}$_\lambda$. Thus, because $\coth(1)\max(1,\sqrt{-\lambda})\geq1$, in either case $\lambda=0$ or $\lambda<0$ we get
\begin{equation*}
  g'(\Delta^2) L(\Delta) \leq 2 A_{g'} \coth(1)\max(1,\sqrt{-\lambda}).
\end{equation*}

Since the right-hand side above is independent of $U$, and $U$ is any bounded, open, convex subset of $\M$, we deduce that for all $x,y,z\in\M$ we have
\begin{align*}
	&\norm{\grad K_z (x) - \Pi_{yx} \grad K_z (y))}_x \leq 2A_{g'} (\coth(1)\max(1,\sqrt{-\lambda}) + 2)\, d(x,y),
\end{align*}
Then we conclude the proof, with Lipschitz constant $\overbar L = 2A_{g'} (\coth(1)\max(1,\sqrt{-\lambda}) + 2)$.
\end{proof}

\begin{lem}\label{lem:grad-lip-2-glob}
  Assume $\K \leq 0$. Suppose $\M$ satisfies \ref{hyp:M} and is unbounded and suppose $K$ satisfies \ref{hyp:Kglob}$_0$. Let furthermore $\rho,\sigma \in \Cont([0,\infty);\P_\infty(\M))$. Then, there exists $\Lip>0$ so that
  \begin{equation*}
    \|\V[\rho]-\V[\sigma]\|_{L^\infty(\M \times [0,\infty))} \leq \Lip \, \bd_1(\rho,\sigma).
  \end{equation*}
\end{lem}
\begin{proof}
  We follow closely the proof of Lemma \ref{lem:grad-lip-2}. By Lemma \ref{lemma:Lipschitz-2} and Remark \ref{rem:negative-Lip}, recall that
\begin{equation}
\label{eq:Lip-d2-proof}
\| \nabla d^2_y(x) - \nabla d^2_z(x) \|_x \leq 2 d(y,z), \qquad \text{ for all } x,y,z \in \M.
\end{equation}
For all $x,y,z\in \M$ compute
	\begin{align*}
		& \|\nabla K_y(x) - \nabla K_z(x)\|_x = \|g'(d(x,y)^2) \nabla d^2_y(x) - g'(d(x,z)^2) \nabla d^2_z(x)\|_x\\[3pt]
          &\quad  \leq |g'(d(x,z)^2)| \|\nabla d^2_y(x) - \nabla d^2_z(x)\|_x + \|\nabla d^2_y(x)\|_x|g'(d(x,y)^2) - g'(d(x,z)^2)|\\[3pt]
          &\quad\leq 2A_{g'} \, d(y,z) + 4A_{g'}\abs{d(x,y)-d(x,z)}\\[3pt]
          &\quad \leq 6 A_{g'} \, d(y,z),	
	\end{align*}
where for the first inequality we added and subtracted $g'(d(x,y)^2)  \nabla d^2_z(x)$ and used triangle inequality; we then used \eqref{eq:Lip-d2-proof}, the fact that $\norm{\nabla d^2_y(x)}= 2 d(x,y)$ and Remark \ref{rem:glob-Lip-g} for the second inequality, and the reverse triangle inequality on $d$ for the third inequality. Following the proof of Lemma \ref{lem:grad-lip-2}, we conclude with $\Lambda = 6 A_{g'}$.
\end{proof}

In addition to global well-posedness, the following theorem shows that the support of the global solution is contained in an increasing geodesic ball centred around the initial data.

\begin{thm}[Global Well-Posedness on Nonpositively Curved Manifold]
  \label{thm:well-posedness-glob}
  Assume that there exists $\lambda\leq0$ such that $\M$ satisfies \ref{hyp:M}$_\lambda$. Suppose either that $\M$ is compact and $K$ satisfies \ref{hyp:Kloc}, or that $\M$ is unbounded and $K$ satisfies \ref{hyp:Kglob}$_\lambda$. Let $\rho_0 \in \P_\infty(\M)$. Then, there exists a unique weak solution in $\Cont([0,\infty);\P_\infty(\M))$ starting from $\rho_0$ to the aggregation equation \eqref{eqn:model}. Moreover, there exist a nondecreasing function $R\:[0,\infty) \to \R$ and a point $p\in \supp(\rho_0)$ such that 
  \begin{equation*}
    \supp(\rho_t) \subset B_{R(t)}(p), \qquad \text{for all $t\geq0$}.
  \end{equation*}
\end{thm}
\begin{proof} 
  By Lemma \ref{lem:v-Lip-glob-compact} (if $\M$ is compact) or Lemma \ref{lem:v-Lip-glob} (if $\M$ is unbounded), since $\M$ is convex, the velocity field $\V[\sigma]$ (for any fixed continuous curve $\sigma$ in $\P_\infty(\M)$) satisfies the assumptions of the global Cauchy-Lipschitz theorem (cf. Theorem \ref{thm:Cauchy-Lip-glob}), so that we can infer that $(\V[\sigma],\supp(\rho_0))$ generates a unique global flow defined on $\supp(\rho_0)\times [0,\infty)$. 

Following the same steps as in the proof of Theorem \ref{thm:well-posedness}, using Lemmas \ref{lem:v-Lip-glob-compact} and \ref{lem:grad-lip-2-glob-compact} (in case $\M$ is compact) or Lemmas \ref{lem:v-Lip-glob} and \ref{lem:grad-lip-2-glob} (in case $\M$ is unbounded) in place of Lemmas \ref{lem:v-Lip} and \ref{lem:grad-lip-2}, respectively, yields the well-posedness by observing that the constants $\overbar L$ and $\Lambda$ are now independent of $\supp(\rho_0)$ and its diameter, which allows us to apply an iterative argument indefinitely and obtain a unique solution in $\Cont([0,\infty);\P_\infty(\M))$. If $\M$ is unbounded, the boundedness of the support of the solution and the existence of a nondecreasing ball containing it follows from the iterative construction of the global flow map in the proof of Theorem \ref{thm:Cauchy-Lip-glob}; if $\M$ is compact, this is trivial.
\end{proof}

\section{Additional results}
\label{sect:add-results}

We give here two additional results which can be derived by following very closely the global well-posedness and stability theory established for the model on the sphere in \cite{FePaPa2020}. For this reason, we do not give the proofs here and refer the reader to \cite[Proposition 4.1]{FePaPa2020} (global well-posedness for attractive potentials) and \cite[Theorem 3.8]{FePaPa2020} (stability of solutions) should they wish to see details.

Following \cite{FePaPa2020}, we give the following definition of purely attractive potential: given $x\in \M$, we say that $K$ is \emph{purely attractive} at $x$ if for all normal neighbourhood $U$ of $x$ and $y\in U$ we have
\begin{equation}
  \label{eqn:purely-att}
  \ap{\grad K_y(x),\log_x y}_x \leq 0.
\end{equation}
Consider now the hypothesis below on $K$:
\begin{namedthm}{(Katt)}
  \namedlabel{hyp:Katt}{{\bfseries \upshape (Katt)}}
  $K\:\M\times\M \to \R$ satisfies \ref{hyp:Kloc} with $g$ such that 
  \begin{equation*}
    g'(r^2) \geq0 \quad \text{for all $r<\iM$}.
  \end{equation*}
\end{namedthm}
\noindent By \eqref{eqn:K-gen}, the nonnegativity condition in \ref{hyp:Katt} means that $K$ is purely attractive.

\begin{thm}[Global Well-Posedness for Attractive Potential]
  \label{thm:well-posedness-glob-attr}
  Let $\M$ and $K$ satisfy \ref{hyp:M} and \ref{hyp:Katt}, and let $U\subset B_{r/2}(p)$ for some $p\in\M$ and $r<\iM$ be open and geodesically convex and such that \eqref{eqn:diam-U} holds. Let moreover $\rho_0 \in \P(U)$ and $R>0$ be such that $\supp(\rho_0) \subset B_R(p) \subset U$. Then, there exists a unique weak solution $\rho$ in $\Cont([0,\infty);\P(U))$ starting from $\rho_0$ to the aggregation equation \eqref{eqn:model}; furthermore, $\supp(\rho_t) \subset \overbar{B_R(p)}$ for all $t\geq0$. 
\end{thm}

\begin{thm}[Stability]
  Let $\rho_0,\sigma_0\in\P_\infty(\M)$. Let furthermore $\rho$ and $\sigma$ be weak solutions to \eqref{eqn:model} defined on some time interval $[0,T)$ starting from $\rho_0$ and $\sigma_0$, respectively, as derived in Theorems \ref{thm:well-posedness}, \ref{thm:well-posedness-glob} or \ref{thm:well-posedness-glob-attr}. Then, 
  \begin{equation*}
    W_1(\rho_t,\sigma_t) \leq e^{(\overbar L + \Lambda)t}\, W_1(\rho_0,\sigma_0) \quad \text{for all $t\in[0,T)$},
  \end{equation*}
  where $\overbar L$ and $\Lambda$ are the constants given in the proofs of Lemmas \ref{lem:v-Lip} and \ref{lem:grad-lip-2} in the local, compact and purely attractive cases and of Lemmas \ref{lem:v-Lip-glob} and \ref{lem:grad-lip-2-glob} in the global and unbounded case. 
\end{thm}

We note that when $\M$ and $K$ verify \ref{hyp:M}$_\lambda$ and \ref{hyp:Kglob}$_\lambda$ for some $\lambda\leq0$, and $\M$ is unbounded (i.e., we are in the global and unbounded case and the constant $\overbar L$ is given in the proof of Lemma \ref{lem:v-Lip-glob}), the above theorem illustrates the fact that the larger (in magnitude) the lower bound $\lambda$ on $\K$ is, the faster solutions may spread apart in time.


\appendix
\section{Appendix: basic concepts and terminology from differential geometry}
\label{app:terminology}

All the concepts discussed here are standard and can be found in any graduate differential geometry book. We refer the reader for example to \cite{doCarmo1992,Lee2018,Chavel2006}.

\smallskip
\noindent \textbf{Logarithm map and normal neighbourhoods.}
Given $x\in\M$, there exists $G\subset T_x\M$ open with $0\in G$ such that the exponential map $\exp_x \: G \to U:=\exp_x(G)$ restricted to $G$ is a diffeomorphism. The inverse of $\exp_x$ is the \emph{logarithm map} at $x$ (on $U$), denoted by $\log_x\: U \to G$. We have that $x\in U$ and $U$ is an open subset of $\M$. We call any such $U$ a \emph{normal neighbourhood} of $x$---conversely, any $x\in U$ has a normal neighbourhood. A subset of $\M$ is said to be a \emph{totally normal neighbourhood} if it is a normal neighbourhood of each of its points.

If $r>0$ is such that $\exp_x$ defines a diffeomorphism on the open ball $B_r(0) \subset T_x\M$, then $B_r(x) := \exp_x(B_r(0))$ is open and is called the \emph{open geodesic ball centred at $x$ of radius $r$}; by the previous paragraph, $B_r(x)$ is by definition a normal neighbourhood of $x$. Also, it coincides with the metric ball centred at $x$ with radius $r$, that is, with the set $\{y\in\M\st d(x,y)<r\}$. The largest such radius $r$ is called the \emph{injectivity radius at $x$} and is denoted $r_\mathrm{inj}(x)$. The infimum over all $y \in \M$ of $r_\mathrm{inj}(y)$ is referred to as the \emph{injectivity radius of $\M$}, which we denote by $\iM$. 

\smallskip
\noindent \textbf{Geodesics and convexity.}
By definition of the exponential map, any point $y$ in a normal neighbourhood of another point $x$ can be connected to $x$ via a unique minimizing geodesic $\gamma_{xy}\:[0,1] \to \M$, which satisfies, for all $t\in [0,1]$,
\bes
	\gamma_{xy}(t) = \exp_x(t\log_x y).
\ees
This formula indicates that $\exp_x$ transforms the straight line $t\mapsto t\log_x(y)$ in $T_x\M$ into the geodesic connecting $x$ and $y$. 

A \emph{(geodesically) convex} subset of $\M$ is a subset in which each pair of points can be connected by a unique minimizing geodesic which lies entirely within it---when open, it is thus a particular case of a totally normal neighbourhood that minimizing geodesics cannot exit. We say that $\M$ is \emph{(geodesically) convex} if it is a geodesically convex subset of itself. Note that a geodesic ball is not necessarily convex. Nevertheless, given any $x\in\M$, it is a fact that $x$ has a convex geodesic ball as a neighbourhood; we refer to the radius of the largest such ball as the \emph{convexity radius at $x$}, denoted $r_{\mt{conv}}(x)$. The infimum over all $y \in \M$ of $r_\mathrm{conv}(y)$ is referred to as the \emph{convexity radius of $\M$}, which we denote by $\cM$.

\smallskip
\noindent \textbf{Normal charts.} Let $x\in \M$ and $U$ be a normal neighbourhood of $x$. We can construct a chart in a canonical way. Indeed, let $\{E_1(x),\dots,E_\dim(x)\}$ be an orthonormal basis of $T_x\M$ and define $E\:T_x\M \to \R^\dim$ by
\be
\label{eqn:ortho-basis}
	E(v) = (v_1,\dots,v_n), \qquad \mbox{for all $v= v_1E_1(x) + \cdots + v_nE_\dim(x) \in T_x\M$},
\ee
i.e., $E(v)$ gives the vector in $\R^\dim$ whose coordinates are the coordinates of $v$ in $\{E_1(x),\dots,E_\dim(x)\}$. Then, define the map $\varphi\: U \to \R^\dim$ by
\be\label{eqn:normal-chart}
	\varphi(y) = E\circ \log_x(y) \quad \mbox{for all $y\in U$};
\ee
note that $\varphi(x) = 0$. The pair $(U,\varphi)$ induces a local chart of $\M$ containing $x$. Any chart thus constructed is referred to as a \emph{normal chart} (generated by $x$).

\smallskip
\noindent \textbf{Push-forward and coordinates.}
Given a differentiable function $f\: \M \to \R^\dim$, a point $x\in\M$ and $v\in T_x\M$, we call $f_*v := \der f(x)(v) \in \R^\dim$ the \emph{push-forward} of $v$ through $f$. Given any chart $(V,\varphi)$ of $\M$ containing some $x,y \in \M$, we write
\bes
	\left \{ \frac{\p}{\p\varphi^1}(y),\dots,\frac{\p}{\p\varphi^\dim}(y)\right \}
\ees
for the basis of $T_y\M$ defined by 
\bes
	\frac{\p}{\p\varphi^i}(y) = \der \varphi^{-1}(\varphi(y))(e_i), \quad \mbox{that is,} \quad e_i = \der\varphi(y)\left(\frac{\p}{\p\varphi^i}(y)\right) = \varphi_* \frac{\p}{\p\varphi^i}(y),
\ees
where, for all $i\in\{1,\dots,n\}$, $e_i$ is the $i$th vector of the canonical basis of $\R^\dim$. For all $v\in T_y\M$ there exist $v_1,\dots,v_n\in\R$ such that 
\bes
	v = v_1 \frac{\p}{\p\varphi^1}(y) + \cdots + v_n \frac{\p}{\p\varphi^\dim}(y).
\ees
Then, by linearity, 
\bes
	\varphi_*v = \der\varphi(y)(v) = v_1 \der \varphi(y) \left( \frac{\p}{\p\varphi^1}(y) \right) + \cdots + v_n \der \varphi(y) \left( \frac{\p}{\p\varphi^\dim}(y) \right) = v_1 e_1 + \cdots + v_n e_n = (v_1,\dots,v_n).
\ees
Hence $\varphi_*v$ gives the vector in $\R^\dim$ whose coordinates are those of $v$ in $\left( \p/\p\varphi^1(y),\dots,\p/\p\varphi^\dim(y) \right)$.

\bigskip
\noindent \textbf{Euclidean norm and instrinsic distance.} 
We give now useful relations between norms in $\R^\dim$ and intrinsic distances on $\M$. Let $(V,\varphi)$ be a chart of $\M$ and let $Q\subset V$ be compact. Then there exist $c_{\varphi,Q},C_{\varphi,Q}>0$ so that
\be
\label{eqn:lemA-1}
c_{\varphi,Q} \norm{\varphi_*v}_{\R^\dim} \leq \norm{v}_x \leq C_{\varphi,Q} \norm{\varphi_*v}_{\R^\dim}, \qquad \mbox{for all $x\in Q$ and $v\in T_x\M$}.
\ee
If furthermore $\varphi(U)\subset\R^\dim$ is convex, then there exists $L_{\varphi,Q}>0$ so that
\be
\label{eqn:lemA-3}
d(x,y) \leq L_{\varphi,Q} \norm{\varphi(x)-\varphi(y)}_{\R^\dim}, \qquad \mbox{for all $x,y\in Q$}.
\ee

We now prove these relations. Let us first show \eqref{eqn:lemA-1}. Define
\bes
T_\varphi = \{(x,v) \in Q\times T\M \st v\in T_x\M,\, \norm{\varphi_*v}_{\R^\dim} = 1\},
\ees
which is compact as well. Then, because $(x,v) \mapsto \norm{v}_x$ is continuous on $T_\varphi$ and we have $\norm{v}_x > 0$ for all $(x,v)\in T_\varphi$, there exist $c_{\varphi,Q},C_{\varphi,Q}>0$ such that
\be
c_{\varphi,Q} \leq \norm{v}_x \leq C_{\varphi,Q}, \qquad \mbox{for all $(x,v)\in T_\varphi$}.
\ee
Fix now $(x,v)\in Q\times T\M$ with $v\in T_x\M$ and $\norm{\varphi_*v}_{\R^\dim}\neq0$. Using the inequality above we get
\be
\label{eqn:ip-proof-2}
\norm{v}_x = \norm{\varphi_*v}_{\R^\dim} \norm{\frac{v}{\norm{\varphi_*v}_{\R^\dim}}}_x \leq C_{\varphi,Q}\norm{\varphi_*v}_{\R^\dim},
\ee
and similarly for the side with $c_{\varphi,Q}$. Also, trivially, if $\norm{\varphi_*v}_{\R^\dim}=0$, then the same inequalities hold, proving \eqref{eqn:lemA-1}.

Let us now turn to \eqref{eqn:lemA-3}. Define
\bes
R_{\varphi,Q} = \{ \xi \in \R^\dim \st \exists\, (x,y,t)\in Q \times Q \times [0,1],\, \xi = (1-t)\varphi(x) + t\varphi(y) \},
\ees
which we note is compact. Let $x,y \in Q$ and define $\gamma$ by
\bes
\gamma(t) = \varphi^{-1}((1-t)\varphi(x) + t\varphi(y)), \qquad \mbox{for all $t\in[0,1]$},
\ees
so that in particular, the composition $\varphi\circ \gamma$ is the (Euclidean) geodesic from $\varphi(x)$ to $\varphi(y)$. Since $\varphi(V)$ is convex we deduce that $R_{\varphi,Q} \subset \varphi(V)$ and so $\gamma([0,1]) \subset \varphi^{-1}(R_{\varphi,Q}) \subset U$. Since $\varphi^{-1}(R_{\varphi,Q})$ is furthermore compact, using \eqref{eqn:ip-proof-2} and setting $L_{\varphi,Q} = C_{\varphi,\varphi^{-1}(R_{\varphi,Q})}$ yields
\begin{align*}
  d(x,y) &\leq \int_0^1 \norm{\gamma'(t)}_{\gamma(t)} \d t \leq L_{\varphi,Q} \int_0^1 \norm{\varphi_*\gamma'(t)}_{\R^\dim} \d t\\
         &= L_{\varphi,Q} \int_0^1 \norm{(\varphi\circ\gamma)'(t)}_{\R^\dim} \d t = L_{\varphi,Q}\norm{\varphi(x) - \varphi(y)}_{\R^\dim}. \qedhere
\end{align*}

\section{Appendix: proofs of preliminary results}
\label{app:prel}

\subsection{Proof of Lemma \ref{lem:pf-Lip}}
\label{appendix:pf-Lip}

Let us first show the following lemma, which is also going to be useful in Appendix \ref{appendix:bHessian-local}.
\begin{lem}
  \label{lem:Lip-connection}
  Let $U\subset \M$ be open and geodesically convex, and let further $X$ be a differentiable vector field on $U$. Then,
  \begin{equation*}
    \norm{X(x) - \Pi_{yx}X(y)}_x^2 \leq \int_0^1 \norm{\conn_{\gamma'(t)} X(\gamma(t))}_{\gamma(t)}^2 \d t, \qquad \text{for all $x,y\in U$},
  \end{equation*}
  where $\gamma\:[0,1] \to \M$ is the minimizing geodesic joining $x$ to $y$. (Refer to \eqref{eqn:cov-der} for the definition of the above integrand.)
\end{lem}
\begin{proof}
  Take two points $x,y \in U$ and the unique minimizing geodesic $\gamma\:[0,1] \to \M$ joining the two points; in particular, $\gamma([0,1])\subset U$ for all $t \in [0,1]$. Consider an orthonormal parallel frame $t\mapsto \{ e_1(t),e_2(t), \dots, e_\dim(t)\}$ along $\gamma$. From
\[
  X(\gamma(t)) = \sum_{i=1}^\dim \langle X(\gamma(t)), e_i(t) \rangle_{\gamma(t)} e_i(t), \qquad \text{for all $t\in[0,1]$},
\]
using that the parallel transport is a linear operator, we get:
\begin{align*}
  \Pi_{yx} X (\gamma(1)) &= \sum_{i=1}^\dim \langle X(\gamma(1)), e_i(1) \rangle_{\gamma(1)} \Pi_{yx} e_i(1) = \sum_{i=1}^\dim \langle X(\gamma(1)), e_i(1) \rangle_{\gamma(1)} e_i(0).
\end{align*}
Consequently, 
\begin{equation}
\label{eqn:diff-coord}
\| X(x) - \Pi_{yx} X(y) \|_x^2 = \sum_{i=1}^\dim  \left( \langle X(\gamma(1)), e_i(1) \rangle_{\gamma(1)} -  \langle X(\gamma(0)), e_i(0) \rangle_{\gamma(0)} \right)^2.
\end{equation}
Using the compatibility of the connection with the metric \cite{doCarmo1992,DuJiMoSaWa2020} we obtain
\begin{align*}
\frac{\der}{\der t} \langle X(\gamma(t)), e_i(t) \rangle_{\gamma(t)} & = \langle \conn_{\gamma'(t)} X(\gamma(t)), e_i(t) \rangle_{\gamma(t)} + \langle  X(\gamma(t)), \conn_{\gamma'(t)} e_i(t) \rangle_{\gamma(t)}\\[2pt]
  &=  \langle \conn_{\gamma'(t)} X(\gamma(t)), e_i(t) \rangle_{\gamma(t)},
\end{align*}
where for the second equality we used  $\conn_{\gamma'(t)} e_i(t) = 0$, for all $t\in[0,1]$. By integrating the equation above from $0$ to $1$ and using it in \eqref{eqn:diff-coord} we find
\begin{align*}
\| X(x) - \Pi_{yx} X(y) \|_x^2 &= \sum_{i=1}^\dim \left( \int_0^1  \langle \conn_{\gamma'(t)} X(\gamma(t)), e_i(t) \rangle_{\gamma(t)} \d t  \right)^2 \\[2pt]
& \leq \int_0^1  \sum_{i=1}^\dim \langle \conn_{\gamma'(t)} X(\gamma(t)), e_i(t) \rangle_{\gamma(t)}^2 \d t\\[2pt]
& = \int_0^1 \| \conn_{\gamma'(t)} X(\gamma(t)) \|_{\gamma(t)}^2 \d t,
\end{align*}
where for the second line we used Jensen's inequality, and for the third we used the fact that $\{e_i(t)\}_{i=1}^\dim$ is an orthonormal basis of $T_{\gamma(t)}\M$.
\end{proof}

  Let us now turn to the proof of Lemma \ref{lem:pf-Lip}. Take $Q\subset B_\delta(x) \cap V$ compact and let $y,z\in Q$. We write $G_\varphi(y)$ for the metric tensor associated with $\ap{\cdot,\cdot}_y$ in the basis $\{\p/\p\varphi^1(y),\dots,\p/\p\varphi^\dim(y)\}$. For all $i,j\in\{1,\dots,\dim\}$ we denote by $p_i(y,z)$ and $f_{ij}(y,z)$ the $i$th coordinates in this basis of $X(y) - \Pi_{zy}X(z)$ and $\p/\p\varphi^j(y) - \Pi_{zy}\p/\p\varphi^j(z)$, respectively. 
  
  For all $i\in\{1,\dots,\dim\}$ we have
  \begin{align*}
    |p_i(y,z)| &= \left| \ap{X(y) - \Pi_{zy}X(z), G_\varphi(y)^{-1} \frac{\p}{\p\varphi^i}(y)}_y \right|\\[3pt]
               &\leq \norm{X(y) - \Pi_{zy}X(z)}_y \norm{G_\varphi(y)^{-1} \frac{\p}{\p\varphi^i}(y)}_y \leq P_\varphi \norm{X}_{\mathrm{Lip}(Q)} d(y,z),
  \end{align*}
  where $P_\varphi = \sup_{i\in\{1,\dots,\dim\}} \norm{G_\varphi^{-1}(\cdot) \p/\p\varphi^i(\cdot)}_{L^\infty(Q)}$, which is finite since $\p/\p\varphi^i$ for all $i\in\{1,\dots,\dim\}$ and $G_\varphi$ are continuous over the compact set $Q$; note that $P_\varphi$ is independent of $y$ and $z$. 
  
 Similarly, for all $i,j\in\{1,\dots,\dim\}$, we have
  \bes
  |f_{ij}(y,z)| \leq P_\varphi F_\varphi d(y,z),
  \ees
  where $F_\varphi = \sup_{i\in\{1,\dots,\dim\}} \norm{\p/\p\varphi^i}_{\mathrm{Lip}(Q)}$, which is finite since $\p/\p\varphi^i$ as a vector field on $B_\delta(x)$ is Lipschitz continuous for all $i\in\{1,\dots,\dim\}$. Indeed, by Lemma \ref{lem:Lip-connection}, for all $i\in\{1,\dots,\dim\}$ there holds
  \begin{equation*}
    \norm{\frac{\p}{\p\varphi^i}(y) - \Pi_{zy}\frac{\p}{\p\varphi^i}(z)}_y^2 \leq \int_0^1 \norm{\conn_{\gamma'(t)} \frac{\p}{\p\varphi^i}(\gamma(t))}_{\gamma(t)}^2 \d t, \qquad \text{for all $y,z\in B_\delta(x)$},
  \end{equation*}
  where $\gamma\:[0,1]\to \M$ is the unique minimizing geodesic linking $y$ to $z$. Decomposing $\conn_{\gamma'(t)} \frac{\p}{\p\varphi^i}(\gamma(t))$ using Christoffel symbols and using that the terms involved are continuous, hence locally bounded, one can bound the integrand in the right-hand side by $L \norm{\gamma'(t)}_{\gamma(t)}$, with $L>0$ constant. Then,  by rescaling $\gamma$ to be constant-speed (i.e., so that $\norm{\gamma'(t)}_{\gamma(t)} = d(y,z)$ for all $t\in[0,1]$), we get the desired Lipschitz continuity:
  \begin{equation*}
    \norm{\frac{\p}{\p\varphi^i}(y) - \Pi_{zy}\frac{\p}{\p\varphi^i}(z)}_y^2 \leq L \int_0^1 \norm{\gamma'(t)}_{\gamma(t)}^2 \d t = L\,d(y,z)^2, \qquad \text{for all $y,z\in B_\delta(x)$}.
  \end{equation*}

Now, write $(x_1(y),\dots,x_n(y)) = \varphi_*X(y) \in \R^n$, so that 
  \bes
  X(y) = x_1(y) \frac{\p}{\p\varphi^1}(y) + \cdots + x_n(y) \frac{\p}{\p\varphi^\dim}(y),
  \ees
  and analogously for $z$ (cf. Appendix \ref{app:terminology}). By linearity of the parallel transport from $z$ to $y$, we have
  \begin{align*}
    \sum_{i=1}^\dim p_i(y,z) \frac{\p}{\p\varphi^i}(y) &= X(y) - \Pi_{zy}X(z) = \sum_{i=1}^\dim x_i(y) \frac{\p}{\p\varphi^i}(y) - \sum_{i=1}^n x_i(z) \Pi_{zy}\frac{\p}{\p\varphi^i}(z)\\
                                                       &= \sum_{i=1}^\dim (x_i(y) - x_i(z)) \frac{\p}{\p\varphi^i}(y) + \sum_{i=1}^\dim x_i(z) \left( \frac{\p}{\p\varphi^i}(y) - \Pi_{zy}\frac{\p}{\p\varphi^i}(z) \right)\\
                                                       &= \sum_{i=1}^\dim (x_i(y) - x_i(z)) \frac{\p}{\p\varphi^i}(y) + \sum_{i=1}^\dim \sum_{j=1}^\dim x_j(z) f_{ij}(y,z) \frac{\p}{\p\varphi^i}(y),
  \end{align*}
  which yields, for all $i\in\{1,\dots,\dim\}$,
  \bes
  p_i(y,z) = x_i(y) - x_i(z) + \sum_{j=1}^\dim x_j(z) f_{ij}(y,z).
  \ees
  Thus, for all $i\in\{1,\dots,\dim\}$,
  \bes
  |x_i(y) - x_i(z)| \leq |p_i(y,z)| + \sum_{j=1}^\dim |x_j(z)| |f_{ij}(y,z)| \leq P_\varphi ( \norm{X}_{\mathrm{Lip}(Q)} + \dim P_\varphi F_\varphi \norm{X}_{L^\infty(Q)}) d(y,z). 
  \ees
  To conclude the proof we use \eqref{eqn:lemA-3}. Indeed, by definition of normal chart, we have $\varphi(B_\delta(x))=\varphi(\exp_x(B_\delta(0)) = E(\log_x(\exp_x(B_\delta(0)))) = E(B_\delta(0)) \cong B_\delta(0)$, which is a convex subset of $\R^\dim$, where $E$ is given in \eqref{eqn:ortho-basis}.

\subsection{Proof of Lemma \ref{lemma:bHessian-local}}
\label{appendix:bHessian-local}

Assume first that $\grad f$ is locally Lipschitz continuous. Let $Q\subset U$ be compact, take any $x\in Q$ and $v \in T_x \M$ and write $\gamma$ for the constant-speed geodesic with $\gamma(0) = x$ and $\gamma'(0) = v$. In particular, $d(\gamma(t),x) = t \|v\|_x$ for all $t\geq0$ such that $\gamma(t)$ is defined, and for $\tau>0$ small enough we have $\gamma(t)\in \widetilde Q$ for all $t<\tau$ for some $\widetilde Q$ compact satisfying $Q\subset \widetilde Q \subset U$. By local Lipschitz continuity of $\grad f$ we thus have, for all $t<\tau$,
\begin{equation}
\label{eqn:Ls-pt}
\| \Pi^{-1}_{x \gamma(t)} \grad f(\gamma(t)) - \grad f(x) \|_x \leq L_{\widetilde Q}\, d(\gamma(t),x) = L_{\widetilde Q}\, t \|v\|_x.
\end{equation}
By \eqref{eqn:cov-der} and the definition of the Hessian, we have
\[
\Hess_vf(x) = \conn_v (\grad f)(x) = \lim_{t \to 0}  \frac{\Pi^{-1}_{x \gamma(t)} \grad f(\gamma(t)) - \grad f(x)}{t},
\]
and using \eqref{eqn:Ls-pt} we find
\[
\| \Hess_vf(x) \|_x \leq L_{\widetilde Q}\, \|v\|_x.
\]
The Hessian operator $\Hess f(x)$ is therefore bounded. By arbitrariness of $Q$ and $x\in Q$, we conclude that $\Hess f$ is locally bounded.

To prove the converse, suppose that $U$ is convex and that the Hessian of $f$ is locally bounded. Let $Q\subset U$ be compact; because $U$ is open and convex, there exists a compact, geodesically convex set $\widetilde Q$ with $Q\subset \widetilde Q \subset U$. Furthermore, take two points $x,y \in Q$ and the constant-speed geodesic $\gamma\:[0,1] \to \M$ joining the two points; in particular, $\| \gamma'(t)\|_{\gamma(t)} = d(x,y)$ and $\gamma([0,1])\subset \widetilde Q$ for all $t \in [0,1]$. By Lemma \ref{lem:Lip-connection} applied to the vector field $\grad f$ we get
\begin{equation*}
  \| \grad f (x) - \Pi_{yx} \grad f (y) \|_x^2 \leq \int_0^1 \| \Hess_{\gamma'(t)}f(\gamma(t)) \|_{\gamma(t)}^2 \d t.
\end{equation*}
By the local boundedness of $\Hess f$ and the fact that $\gamma$ is constant-speed we further get
\begin{align*}
  {\| \grad f (x) - \Pi_{yx} \grad f (y) \|}_x^2 &\leq C_{\widetilde Q}^2\, d(x,y)^2,
\end{align*}
for some $C_{\widetilde Q}>0$, which, by arbitrariness of $Q$ and $x,y\in Q$, ends the proof.

\subsection{Proof of Theorem \ref{thm:Cauchy-Lip}}
\label{appendix:Cauchy-Lip}



Take $x\in U$ and let $\delta\leq r_{\mt{conv}}(x)$ be such that $B_\delta(x)\subset U$. Let $(B_\delta(x),\varphi)$ be a normal chart generated by $x$, and consider the initial-value problem
	\be\label{eqn:characteristics-Rn}
		\begin{cases} \alpha'(t) = \Xi(\alpha(t),t),\\ \alpha(0) = \varphi(x),
	\end{cases}
	\ee
	where we define $\Xi\: \varphi(B_\delta(x)) \times [0,T) \to \R^\dim$ by
	\bes
		\Xi(\xi,t) = \varphi_*X_t (\varphi^{-1}(\xi)), \qquad \mbox{for all $(\xi,t) \in \varphi(B_\delta(x)) \times [0,T)$}.
	\ees
	Take $R\subset \varphi(B_\delta(x)) = B_\delta(0)$ and $S\subset [0,T)$ compact, so that in particular $Q:=\varphi^{-1}(R)\subset B_\delta(x)$ is compact. For all $\xi,\eta\in R$ and $t\in S$, our Lipschitz-continuity assumption on $X_t$ yields
	\begin{align*}
		\norm{\Xi(\xi,t) - \Xi(\eta,t)}_{\R^\dim} &= \norm{\varphi_*X_t (\varphi^{-1}(\xi)) - \varphi_*X_t (\varphi^{-1}(\eta))}_{\R^\dim}\\
		&\leq (A \norm{X_t}_{L^\infty(Q)} + B \norm{X_t}_{\mt{Lip}(Q)}) \norm{\xi-\eta}_{\R^\dim},
	\end{align*}
	for some constants $A,B>0$ coming from Lemma \ref{lem:pf-Lip}. Also, for all $\xi\in R$ it holds that
	\bes
		\norm{\Xi(\xi,t)}_{\R^\dim} = \norm{\varphi_*X_t (\varphi^{-1}(\xi))}_{\R^\dim} \leq C \norm{X_t}_{L^\infty(Q)},
	\ees
        where $C>0$ comes from \eqref{eqn:lemA-1}. Therefore, by \eqref{eqn:bdd-Lip} we get that $\Xi$ satisfies
	\bes
		\int_S \left( \norm{\Xi(\cdot,t)}_{L^\infty(R)} + \norm{\Xi(\cdot,t)}_{\mathrm{Lip}(R)} \right) \d t \leq \int_S \left( (A+C) \norm{X_t}_{L^\infty(Q)} + B \norm{X_t}_{\mathrm{Lip}(Q)} \right) \d t < \infty.
	\ees
	
	By arbitrariness of the compact sets $R\subset \varphi(B_\delta(x))$ and $S\subset [0,T)$ and by the classical Cauchy--Lipschitz theorem on $\R^\dim$ (see for instance \cite[Lemma 8.1.4]{AGS2005}), this yields the existence of a unique maximal solution $\alpha_x$ to \eqref{eqn:characteristics-Rn} defined on some time interval $[0,\tau_x)$, with $\tau_x\leq T$, and with values in $\varphi(B_\delta(x))$. By defining $\varphi_x = \varphi^{-1} \circ \alpha_x$, we see that $\alpha_x$ satisfies \eqref{eqn:characteristics-Rn} if and only if
	\be\label{eqn:characteristics-Rn-chart}
		\begin{cases} \varphi_*\varphi_x'(t) = (\varphi \circ \varphi_x)'(t) = \varphi_*X_t(\varphi_x(t)) & \mbox{for all $t\in [0,\tau_x)$},\\ \varphi(\varphi_x(0)) = \varphi(x).
	\end{cases}
	\ee
	By the bijectivity of $\varphi$, we get that $\varphi_x$ is thus the unique maximal solution to the characteristic equation starting at $x$. 
	
	Let now $\Sigma$ be a compact subset of $U$. We are left with showing that $\tau:=\inf_{x\in\Sigma}(\tau_x)>0$. By classical Euclidean Lipschitz theory, we deduce that for all $x\in U$ there exists $\delta_x>0$ such that $\bar\tau_x:=\inf_{y\in B_{\delta_x}(x)} \tau_y > 0$. Since $\Sigma$ is compact we know it can be covered by a finite subfamily of $\{B_{\delta_x}(x)\}_{x\in\Sigma}$, which we index by $\{x_1,\dots,x_n\}$ for some $n\in\mathbb{N}$. We thus get $\tau = \min_{i\in\{1,\dots,n\}} \bar\tau_{x_i} > 0$. The map $\Psi$ defined by $\Psi(x,t) = \varphi_x(t)$ for all $x\in \Sigma$ and $t\in [0,\tau)$ is then the unique maximal flow map generated by $(X,\Sigma)$.

\subsection{Proof of Theorem \ref{thm:Cauchy-Lip-glob}}
\label{appendix:Cauchy-Lip-glob}

We distinguish two cases: $\M$ is compact, and $\M$ is unbounded. 

The first case is direct by the Escape Lemma (cf. \cite[Exercise 4.10]{Lee2018} for instance), which states that local solutions cannot be contained in any compact subset of $\M$. The local Cauchy--Lipschitz theorem (cf. Theorem \ref{thm:Cauchy-Lip}) together with a contradiction argument suffices then to end this case.

Suppose now that $\M$ is unbounded. First note that the bound assumption on the Lipschitz constants of our vector fields implies that there exists $C>0$ such that for any $U\subset\M$ we have
  \begin{equation}
    \label{eqn:linear-bound}
    \sup_{t\in[0,\infty)} \norm{X_t}_{L^\infty(U)} \leq C(1+\Diam(U)).
  \end{equation}
Let $\Sigma\subset\M$ be compact and $U$ be bounded and open and such that $\Sigma\subset U$. Since $\M$ is convex we know $U$ is a totally normal neighborhood. Denote by $\widetilde X$ the restriction of $X$ to $U\times [0,\infty)$. By the local Cauchy--Lipschitz theorem (cf. Theorem \ref{thm:Cauchy-Lip}), there exists a unique maximal flow map generated by $(\widetilde X,\Sigma)$ associated with some maximal time of existence $\tau>0$. Recalling the proof of Theorem \ref{thm:Cauchy-Lip} and the classical Cauchy--Lipschitz theorem in Euclidean setting, we know there exists a constant $\alpha>0$ such that 
  \begin{equation*}
    \tau\geq \alpha \min\left(\frac{\dist(\Sigma,\p U)}{\sup_{t\in[0,\infty)}\norm{X_t}_{L^\infty(U)}},\frac{1}{\sup_{t\in[0,\infty)} \norm{X_t}_{\mt{Lip}(\M)}}\right),
  \end{equation*}
where $\p U$ is the boundary of $U$. By \eqref{eqn:linear-bound} we then have
\begin{equation*}
    \tau\geq \alpha \min\left(\frac{\dist(\Sigma,\p U)}{C(1+\Diam(U))},\frac{1}{\sup_{t\in[0,\infty)} \norm{X_t}_{\mt{Lip}(\M)}}\right).
  \end{equation*}
Because $\M$ is unbounded, we see that we can always choose $U$ a posteriori so that $\tfrac{\dist(\Sigma,\p U)}{1+\Diam(U)} \geq 1$. Overall, this makes our lower bound on the maximal time of existence independent of $\Sigma$ and $U$. Applying a simple iterative argument extends our restricted local flow map to a unique global one, thus showing the desired result.

\subsection{Proof of Lemma \ref{lemma:d2-Hessian}}
\label{appendix:d2-Hessian}

The proof follows from \eqref{eqn:bHessian}. Indeed, we can apply \eqref{eqn:bHessian} since the bounds on the diameter of $U$ ensure that, given $z\in U$, the geodesic ball $B_r(z)$ contains $U$ and is contained in the set $\{x\in\M \st d(p,x) < 2r \}$, 
where the sectional curvature is bounded below by $\lambda$ and above by $\mu$, so that \eqref{eqn:K-bounds} holds on $B_r(z)$. Thus, \eqref{eqn:bHessian} yields
\begin{equation*}
  \norm{\Hess_v d_z^2(x)}_x \leq 2 \sqrt{-\lambda}\, d(x,z) \coth( \sqrt{-\lambda}\, d(x,z)) \|v\|_x \leq L \|v\|_x , 
\end{equation*}
for all $x,z\in U \subset B_r(z)$, where we used that the function $s\mapsto s \coth(s)$ is nondecreasing and the fact that the left-hand side in \eqref{eqn:bHessian} is nonnegative (when $\mu>0$, use $\sqrt{\mu}\, d(x,z) \leq \sqrt{\mu}\, \Delta \leq \frac{\pi}{2}$). Hence the Hessian $\Hess d_z^2$ is bounded by $L$ on $U$.

If $U$ is now convex, then \eqref{eqn:d2-Lipschitz} follows directly from Lemma \ref{lemma:bHessian} and \eqref{eqn:grad-d2}.

\subsection{Proof of Lemma \ref{lemma:Lipschitz-2}}
\label{appendix:Lipschitz-2}

For completeness, let us give the statement of \cite[Corollary 6.6.1]{Jost2017}, which lies at the core of the proof of Lemma \ref{lemma:Lipschitz-2}. Given a differentiable function $f\: M_1\to M_2$ between two smooth Riemannian manifolds $\M_1$ and $\M_2$, we write $D_x f \: T_x M_1 \to T_{f(x)}\M_2$ for the differential map of $f$ at $x$.

\begin{lem}
  \label{lem:bounds-exp}
  Let $p\in\M$ and $r<\iM$. Let there be a minimizing geodesic connecting two distinct points $x$ and $y$ in $B_r(p)$, with velocity $u$ at the point $x$. Write $\mu = \mu_r(p)$, and suppose that $\norm{u}_x \leq \frac{\pi}{\sqrt{\mu}}$ in case $\mu>0$, and let $w\in T_x\M$ be so that $\ap{w,u}_x = 0$. Then,
  \begin{equation}
    \label{eqn:bounds-exp}
    \norm{w}_x \frac{s_\mu (\norm{u}_x)}{\norm{u}_x} \leq \norm{D_{u}\exp_x(w)}_y,
  \end{equation}
  where $s_\mu$ is defined, for all $a\geq0$, by
  \begin{equation*}
    s_\mu(a) = \begin{cases} \frac{\sin(a\sqrt{\mu})}{\sqrt{\mu}} & \text{if $\mu > 0$,}\\ a & \text{if $\mu=0$.} \end{cases}
  \end{equation*}
\end{lem}

Let us now prove Lemma \ref{lemma:Lipschitz-2}. The case $\mu=0$ is stated in \cite[Corollary 6.9.1]{Jost2017}. We give its proof here for completeness. Fix $x,y \in U$ with $x\neq y$. Because $U$ is a totally normal neighborhood there exists only one minimizing geodesic linking $x$ to $y$ with velocity $\log_xy$ at the origin, and the exponential map $\exp_x$ is a diffeomorphism from $T_x\M$ to $U$. Then, we can let $v \in T_y \M$ and apply Lemma \ref{lem:bounds-exp} with $w = D_y \log_x (v)$ and $u = \log_x y$. Note that since $\log_x \:\M \to T_x \M$, we have $D_y \log_x(v) \in T_{\log_x y} T_x \M \cong T_x \M$. Also, up to a rotation, we can assume $\ap{\log_x y,D_y \log_x(v)}_x = 0$. By \eqref{eqn:bounds-exp}, we then get
\begin{equation}
\label{eqn:w-ineq}
\| w \|_x \leq \| D_{\log_x y} \exp_x (w)\|_y = \| D_{\log_x y} \exp_x (D_y \log_x (v))\|_y = \|v\|_y,
\end{equation}
where for the last equality we used $\exp_x \circ \log_x$ is the identity map on $U$. We give the obtained result as:
\begin{equation}
\label{eqn:Dlog-ineq}
{\| D_y \log_x (v) \|}_x \leq {\| v \|}_y, \qquad \text{ for all } x,y \in \M, v \in T_y \M.
\end{equation}
Now fix $x,y,z \in U$, and denote by $\gamma$ the constant-speed geodesic between $y$ and $z$. Then,
\begin{align}
\log_x z - \log_x y &= \int_0^1 \frac{\der}{\der t} \log_x \gamma(t)\d t \nonumber \\
& =  \int_0^1 D_{\gamma(t)} \log_x (\gamma'(t)) \d t. \label{eqn:log-diff-ftc}
\end{align}
By applying \eqref{eqn:Dlog-ineq} in the integral above for $y=\gamma(t)$ and $v = \gamma'(t)$, we get from \eqref{eqn:log-diff-ftc}:
\begin{align}
{\| \log_x z - \log_x y \|}_x &\leq \int_0^1 {\| D_{\gamma(t)} \log_x (\gamma'(t)) \|}_x \d t \nonumber \\[2pt]
&\leq \int_0^1 \| \gamma'(t) \|_{\gamma(t)} \d t, \label{eqn:log-diff}
\end{align}
which leads to \eqref{eqn:log-diff-mu0} given that $\|\gamma'(t)\| = d(y,z)$ for all $t\in[0,1]$.
  
Let us turn to the proof for the case $\mu>0$, which is again based on Lemma \ref{lem:bounds-exp}. Fix $x,y \in U$ with $x\neq y$. Similar to the $\mu=0$ case, there exists only one minimizing geodesic joining $x$ and $y$ with velocity $\log_xy$ at $x$, and the exponential map $\exp_x$ is a diffeomorphism from $T_x\M$ to $U$. By the assumption on the diameter of $U$ we have $d(x,y) \leq \frac{\pi -\e}{\sqrt{\mu}} < \frac{\pi}{\sqrt{\mu}}$, and hence we can then let $v \in T_y \M$ and apply \eqref{eqn:bounds-exp} for $w = D_y \log_x (v)$ and $u = \log_x y$, to get
\begin{equation}
\label{eqn:w-ineq-mu}
\| w \|_x \leq \frac{\sqrt{\mu}\|\log_x y\|_x}{\sin (\sqrt{\mu}\|\log_x y\|_x)} {\| D_{\log_x y} \exp_x (w)\|}_y =  \frac{\sqrt{\mu}\|\log_x y\|_x}{\sin (\sqrt{\mu} \|\log_x y\|_x)} \|v\|_y,
\end{equation}
and hence,
\begin{equation}
\| D_y \log_x (v) \|_x \leq  \frac{\sqrt{\mu}\| \log_x y\|}{\sin (\sqrt{\mu}\|\log_x y\|)} \| v \|_y.
\end{equation}
Since the function $\tau \mapsto\frac{\tau}{\sin \tau}$ is nondecreasing on $[0,\pi - \e]$, we get from above that
\begin{equation}
\label{eqn:Dlog-ineq-mu}
\| D_y \log_x (v) \|_x \leq  \frac{\pi - \e}{\sin(\pi - \e)} \| v \|_y, \qquad \text{ for all } x,y \in U, \text{ and } v \in T_y \M.
\end{equation}
Now fix $x,y,z \in U$ and take $\gamma$ the constant-speed geodesic between $y$ and $z$. Restarting from \eqref{eqn:log-diff-ftc}, equation \eqref{eqn:Dlog-ineq-mu} yields (see also \eqref{eqn:log-diff})
\[
\| \log_x z - \log_x y \|_x \leq  \frac{\pi - \e}{\sin(\pi - \e)} \int_0^1 \| \gamma'(t) \|_{\gamma(t)} \d t.
\]
The Lipschitz estimate \eqref{eqn:log-diff-mu} now follows from above, using that $\| \gamma'(t)\| = d(y,z)$ for all $t \in [0,1]$. Note that the Lipschitz constant approaches $\infty$ as $\e \to 0$. 

\bigskip
\subsubsection*{Acknowledgements}
The authors thank Xavier Pennec for his time and insightful discussions on the topic of bounded curvature.

\bibliographystyle{abbrv}

\def\cprime{$'$}

\end{document}